\documentclass[a4paper, 12pt]{amsart}
\usepackage{amsmath,amscd,amsfonts,amssymb}
\usepackage{mathrsfs,dsfont}
\usepackage{tikz}

\numberwithin{equation}{section}
\theoremstyle{definition}
\newtheorem*{remark*} {Remark}

\theoremstyle{plain}
\newtheorem*{theorem*}{Theorem}
\newtheorem{theorem}{Theorem}[section]
\newtheorem{lemma}[theorem]{Lemma}
\newtheorem{proposition}[theorem]{Proposition}
\newtheorem{corollary}[theorem]{Corollary}
\newtheorem{definition}[theorem]{Definition}
\newtheorem{remark}[theorem]{Remark}

\newcommand{\bz}{\mathbb{Z}}

\newcommand{\br}{\mathbb{R}}
\newcommand{\bc}{\mathbb{C}}
\newcommand{\bn}{\mathbb{N}}
\newcommand{\bq}{\mathbb{Q}}
\newcommand{\vx}{\mathbf{x}}

\renewcommand{\geq}{\geqslant}
\renewcommand{\leq}{\leqslant}
\newcommand{\norm}[1]{\left\lVert#1\right\rVert}

\newcommand\modu{\operatorname{mod}}

\newcommand{\floor}[1]{\lfloor #1 \rfloor}

\newenvironment{enumerate-math-abc}
{\begin{enumerate}
\addtolength{\itemsep}{5pt}
}
{\end{enumerate}}

\title[Sets of bounded remainder]{Sets of bounded remainder for the continuous irrational rotation on $[0,1]^2$}
\author{Sigrid Grepstad and Gerhard Larcher}
\address{Institute of Financial Mathematics and Applied Number Theory, Johannes Kepler University Linz, Altenbergerstr.\ 69, A-4040 Linz, Austria.}
\email{sigrid.grepstad@jku.at}
\address{Institute of Financial Mathematics and Applied Number Theory, Johannes Kepler University Linz, Altenbergerstr.\ 69, A-4040 Linz, Austria.}
\email{gerhard.larcher@jku.at}

\date{February 29, 2016}
\thanks{The authors are supported by the Austrian Science Fund (FWF): Projects F5505-N26 and F5507-N26, which are both part of the Special Research Program ``Quasi-Monte Carlo Methods: Theory and Applications''.}

\begin{document}

\begin{abstract}
 We study sets of bounded remainder for the two-dimensional continuous irrational rotation $(\{x_1+t\}, \{x_2+t\alpha \})_{t \geq 0}$ in the unit square. In particular, we show that for almost all $\alpha$ and every starting point $(x_1, x_2)$, every polygon $S$ with no edge of slope $\alpha$ is a set of bounded remainder. Moreover, every convex set $S$ whose boundary is twice continuously differentiable with positive curvature at every point is a bounded remainder set for almost all $\alpha$ and every starting point $(x_1, x_2)$. Finally we show that these assertions are, in some sense, best possible. 
\end{abstract}

\maketitle

%%%%%%%%%%%%%%%%%%%%%%%%%%%%%%%%%%%%%%%%%%%% Introduction %%%%%%%%%%%%%%%%%%%%%%%%%%%%%%%%%%%%%%%%%%%%%%%%%

\section{Introduction}
In this paper we will be concerned with bounded remainder sets for the two-dimensional irrational rotation on the unit square $I^2 = [0,1)^2$. 
\begin{definition}
 Let $\vx = (x_1, x_2) \in I^2$, and let $\alpha \in \br \setminus \bq$. We say that the function $X : [0, \infty) \mapsto I^2$ defined by
 \begin{equation*}
  X(t) = \left( \{x_1+t\}, \{x_2+\alpha t \}\right)
 \end{equation*}
is the two-dimensional continuous irrational rotation with slope $\alpha$ and starting point $\vx$.
\end{definition}

\begin{definition}
\label{def:cbrs}
Let $S \subset I^2$ be an arbitrary measurable subset of the unit square with Lebesgue measure $\lambda (S)$. We say that $S$ is a \emph{bounded remainder set} for the continuous irrational rotation with slope $\alpha>0$ and starting point $\vx=(x_1, x_2) \in I^2$ if the distributional error
\begin{equation}
 \label{eq:discrepancy}
 \Delta_T(S, \alpha, \vx) = \int_0^T \chi_S \left( \{x_1+t\}, \{x_2+\alpha t \} \right) \, dt - T \lambda (S)
\end{equation}
is uniformly bounded for all $T > 0$. Here, $\chi_S$ denotes the characteristic function for the set $S$.
\end{definition}
 
Bounded remainder sets have been extensively studied for the discrete analogue of continuous irrational rotation, that is, for Kronecker sequences $(\{n\alpha_1\}, \{n\alpha_2\}, \ldots , \{n\alpha_s\})_{n=1, 2, \ldots}$ in $[0,1)^s$, where $\alpha_1, \ldots , \alpha_s$ are given reals. In this context, a bounded remainder set $S \subseteq [0,1)^s$ is a measurable set for which the difference
\begin{equation*}
\left| \sum_{n=1}^N \chi_S (\{x_1+ n\alpha_1\}, \ldots , \{x_s+n\alpha_s\}) - N \lambda (S) \right|
\end{equation*}
is uniformly bounded for all integers $N\geq 1$ and almost every point $(x_1, \ldots , x_s) \in [0,1)^s$. In the simplest case when $s=1$ and $S$ is just an interval, bounded remainder sets for the Kronecker sequences were explicitly characterized by Hecke \cite{hecke}, Ostrowski \cite{ostrowski1, ostrowski2} and Kesten \cite{kesten}. In the general multi-dimensional case, a characterization of bounded remainder sets in terms of equidecomposability to certain parallelepipeds was recently given in \cite{grlev}.

Without going into further detail on the known results for the Kronecker sequences, let us simply emphasize that in the discrete case, a given set $S \subset [0,1)^s$ is a bounded remainder set for only ``very few'' choices of $(\alpha_1, \ldots , \alpha_s)$. Likewise, given a vector $(\alpha_1, \ldots, \alpha_s)$, the class of sets $S$ which are of bounded remainder with respect to this vector is, in some sense, small.
Once we consider bounded remainder sets for the \emph{continuous} irrational rotation, the situation turns out to be quite different. In light of recent work by J\'{o}szef Beck, this is not entirely unexpected. Beck studied distributional properties of the continuous irrational rotation in \cite{beck1, beck2, beck3}, and showed in particular that:
\begin{theorem*}[Beck {\cite[Theorem 1]{beck3}}]
Let $S \subseteq I^2$ be an arbitrary Lebesgue measurable set in the unit square with positive measure. Then for every $\varepsilon >0$, almost all $\alpha>0$ and every starting point $\mathbf{x}=(x_1, x_2)\in I^2$, we have
\begin{equation*}
 \int_0^T \chi_S \left( \{x_1+t\}, \{x_2+\alpha t\} \right) \, dt - T \lambda (S) = o \left( (\log T)^{3+\varepsilon} \right) .
\end{equation*}
\end{theorem*}
As pointed out by Beck, the polylogarithmic error term is shockingly small compared to the linear term $T \lambda (S)$. Moreover, it holds for \emph{all} measurable sets $S$. It is thus natural to ask if imposing certain regularity conditions on $S$ could give an even lower bound on the error term.

The aim of this paper is to show that the estimate of Beck can be significantly improved for a large collection of sets $S$. We show that:
\begin{theorem}
 \label{thm:mainpoly}
 For almost all $\alpha>0$ and every $\mathbf{x} \in I^2$, every polygon $S \subset I^2$ with no edge of slope $\alpha$ is a bounded remainder set for the continuous irrational rotation with slope $\alpha$ and starting point $\mathbf{x}$.
\end{theorem}
\begin{theorem}
 \label{thm:mainconvex}
 For almost all $\alpha>0$ and every $\mathbf{x} \in I^2$, every convex set $S \subset I^2$ whose boundary $\partial S$ is a twice continuously differentiable curve with positive curvature at every point is a bounded remainder set for the continuous irrational rotation with slope $\alpha$ and starting point $\mathbf{x}$. 
\end{theorem}
We will see from the proofs that Theorems \ref{thm:mainpoly} and \ref{thm:mainconvex} hold for all $\alpha$
whose continued fraction expansion $\alpha=[a_0; a_1, a_2,\cdots]$ satisfies
\begin{equation*}
 \sum_{l=0}^s \frac{a_{l+1}}{q_{l}^{1/2}} \sum_{k=1}^{l+1}a_k < C ,
\end{equation*}
where $C$ is a constant independent of $s$.
Here, $(q_l)_{l \geq 0}$ is the sequence of best approximation denominators for $\alpha$.

The following results are immediate consequences of Theorems \ref{thm:mainpoly} and \ref{thm:mainconvex}.
 \begin{corollary}
 \label{cor:mainpoly}
  Let $S$ be a polygon in $I^2$. Then $S$ is a bounded remainder set with respect to continuous irrational rotation for almost every $\alpha>0$ and every starting point $\mathbf{x} \in I^2$.
 \end{corollary}
 \begin{corollary}
  \label{cor:mainconvex}
  Let $S$ be a convex set in $I^2$ whose boundary $\partial S$ is a twice continuously differentiable curve with positive curvature at every point. Then $S$ is a bounded remainder set with respect to continuous irrational rotation for almost every $\alpha>0$ and every starting point $\mathbf{x} \in I^2$.
 \end{corollary}
In light of Corollaries \ref{cor:mainpoly} and \ref{cor:mainconvex}, it is tempting to raise the question of whether \emph{every} convex set $S \subset I^2$ is a bounded remainder set with respect to continuous irrational rotation for almost every slope $\alpha>0$ and every starting point $\mathbf{x} \in I^2$. We leave this question open. 

 Theorems \ref{thm:mainpoly} and \ref{thm:mainconvex} above are, in a certain sense, optimal. First of all, the slope condition in Theorem \ref{thm:mainpoly} on the edges of the polygon $S$ cannot be omitted. To see this, fix some $\alpha>0$, and let $S$ be the parallelogram shown in Figure \ref{fig:fullpgram} with $p \notin \bz \alpha (\modu 1)$ and $\lambda (S) = p$. 
 \begin{figure}[htb]
 \centering
 \begin{tikzpicture}[scale=5]
%axes
\draw (-.15,0) -- (1.15, 0);
\draw (0,-.1) -- (0, 1.1);
%limits of unit square
\draw[dashed] (0,1) -- (1,1) -- (1,0);
\draw(0,1) node[left]{$1$};
\draw(1,0) node[below]{$1$};
%the set S
\draw[draw=black,fill={black!20!white}]
	(0,0) -- (1, {1/(2*sqrt(2))}) -- (1, {1/(2*sqrt(2))+.4}) -- (0,.4) -- (0,0);
\draw(.5, {.2+1/(4*sqrt(2))}) node{$S$};
%arrows
\draw[latex-latex] (-.1,0) -- (-.1,.4);
\draw(-.1, .2) node[left]{$p$};
\draw[latex-latex] (1.1,0) -- (1.1, {1/(2*sqrt(2))});
\draw(1.1,{1/(4*sqrt(2))}) node[right]{$\alpha$};
\end{tikzpicture}
 \caption{The parallelogram $S$ with two edges of slope $\alpha$. \label{fig:fullpgram}} 
 \end{figure}
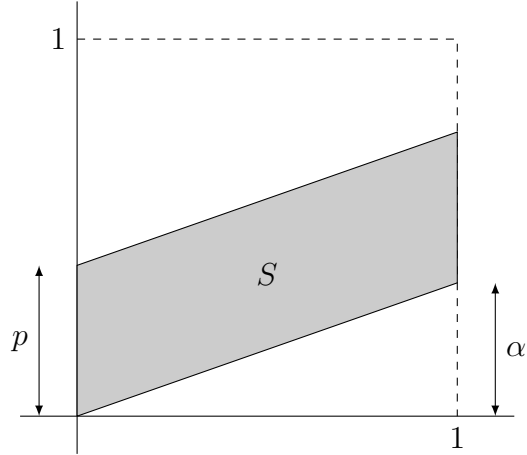
 It is not difficult to show that for such a set $S$, with two edges of slope $\alpha$, we have
 \begin{equation*}
  \left| \int_0^T \chi_S \left( \{t\}, \{\alpha t\} \right) \, dt - \sum_{n=1}^{\floor{T}} \chi_{[0,p)}(\{n\alpha\}) \right| \leq 1.
 \end{equation*}
 We recall from the discrete setting that if $p \notin \bz \alpha (\modu 1)$, then the difference
 \begin{equation*}
  \left| \sum_{n=1}^{\floor{T}} \chi_{[0,p)}(\{n\alpha\}) - p \floor{T} \right|
 \end{equation*}
 is unbounded as $T \rightarrow \infty$ \cite{kesten}, and accordingly so is 
 \begin{equation*}
  \left| \Delta_T(S, \alpha, 0) \right| = \left| \int_0^T \chi_S \left( \{t\}, \{ \alpha t\}\right) - pT\right| .
 \end{equation*}
 Thus, the set $S$ in Figure \ref{fig:fullpgram} is not of bounded remainder for the continuous irrational rotation with slope $\alpha$ starting at the origin. By an equivalent argument, all sets $S'$ similar to the examples shown in Figure \ref{fig:modpgram} with $p \notin \bz \alpha (\modu 1)$ are \emph{not} bounded remainder sets.
 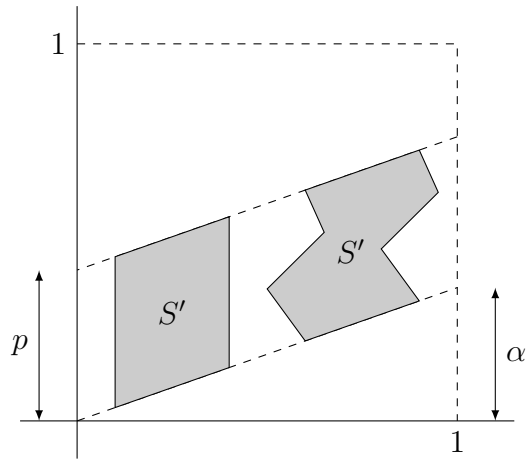
\begin{figure}[htb]
  \centering
 \begin{tikzpicture}[scale=5]
%axes
\draw (-.15,0) -- (1.15, 0);
\draw (0,-.1) -- (0, 1.1);
%limits of unit square
\draw[dashed] (0,1) -- (1,1) -- (1,0);
\draw(0,1) node[left]{$1$};
\draw(1,0) node[below]{$1$};
%dashed band borders
\draw[dashed] (0,0) -- (1, {1/(2*sqrt(2))});
\draw[dashed] (1, {1/(2*sqrt(2))+.4}) -- (0,.4);
%the sets S'
\draw[draw=black,fill={black!20!white}]
	(.1,{.1/(2*sqrt(2))}) -- (.4, {.4/(2*sqrt(2))}) -- (.4, {.4+.4/(2*sqrt(2))}) -- (.1, {.4+.1/(2*sqrt(2))}) -- (.1, {.1/(2*sqrt(2))});
\draw(.25, {.2+.25/(2*sqrt(2))}) node{$S'$};
\draw[draw=black,fill={black!20!white}]
	(.6, {.6/(2*sqrt(2))}) -- (.9, {.9/(2*sqrt(2))}) -- (.8, {.35+.3/(2*sqrt(2))}) -- (.95, {.5+.3/(2*sqrt(2))}) -- (.9, {.4+.9/(2*sqrt(2))}) -- (.6, {.4+.6/(2*sqrt(2))}) --(.65,.5) -- (.5,.35)-- (.6, {.6/(2*sqrt(2))});
\draw(.72, .45) node{$S'$};
%arrows
\draw[latex-latex] (-.1,0) -- (-.1,.4);
\draw(-.1, .2) node[left]{$p$};
\draw[latex-latex] (1.1,0) -- (1.1, {1/(2*sqrt(2))});
\draw(1.1,{1/(4*sqrt(2))}) node[right]{$\alpha$};
\end{tikzpicture}
 \caption{Sets $S'$ which are not of bounded remainder for the continuous irrational rotation with slope $\alpha$ (given $p \notin \bz \alpha (\modu 1)$). \label{fig:modpgram}}
 \end{figure}

Secondly, in neither Theorem \ref{thm:mainpoly} nor \ref{thm:mainconvex} can we replace ``for almost all $\alpha$'' by ``for all irrational $\alpha$''. This is clarified by the following:
\begin{theorem}
 \label{thm:mainneg}
\quad 
\begin{enumerate-math-abc}
 \item \label{it:negpoly} For uncountably many $\alpha>0$ there exist triangles in $I^2$ with no edge of slope $\alpha$ which are not bounded remainder sets for the continuous irrational rotation with slope $\alpha$ and arbitrary starting point.
 \item \label{it:negconvex} For uncountably many $\alpha>0$ there exist discs in $I^2$ which are not bounded remainder sets for the continuous irrational rotation with slope $\alpha$ and arbitrary starting point.
 \item \label{it:specset} The triangle with vertices $(0,0)$, $(1,0)$ and $(0,1)$ is a bounded remainder set for every slope $\alpha >0$ and every starting point $\mathbf{x} \in I^2$. 
\end{enumerate-math-abc}
\end{theorem}
Theorem \ref{thm:mainneg} \eqref{it:specset} illustrates that for very special polygons $S$, Theorem \ref{thm:mainpoly} does actually hold for all irrational $\alpha$. Other trivial examples of such special sets are rectangles of the form $[0,\gamma) \times [0,1)$ (or $[0,1)\times[0,\gamma)$), where $0 < \gamma \leq 1$.

Finally, let us point out that Theorems \ref{thm:mainpoly} and \ref{thm:mainconvex}, and their proofs, give information on the behavior of discrepancies of the continuous irrational rotation on the unit square. Let $\mathcal{B}$ denote a certain class of measurable subsets of $I^2$. Then by the discrepancy $D_T^{(\mathcal{B})}$ of a continuous irrational rotation with slope $\alpha>0$ and starting point $\mathbf{x} \in I^2$ with respect to $\mathcal{B}$ we mean
\begin{equation*}
 D_T^{(\mathcal{B})} := \sup_{S \in \mathcal{B}} \Delta_T (S, \alpha, \mathbf{x}),
\end{equation*}
with $\Delta_T (S, \alpha, \mathbf{x})$ defined in \eqref{eq:discrepancy}. The most extensively studied case in the classical theory of irregular distribution is that when $\mathcal{B}$ is the class of axis-parallel rectangles. Theorem \ref{thm:mainpoly} tells us that in this case, we have 
\begin{equation*}
 \Delta_T (S, \alpha, \mathbf{x}) = O(1)
\end{equation*}
for all $\mathbf{x}$, almost all $\alpha$ and all $S \in \mathcal{B}$. Moreover, by a careful consideration of the constants involved in the proof of Theorem \ref{thm:mainpoly}, one can verify that the $O$-constant will depend only on $\alpha$, and not on the choice of rectangle $S$. As a consequence, we obtain the following result, previously shown by Drmota \cite{drmota} (see also \cite{tichy}).
\begin{corollary}
\label{cor:discr1}
 The discrepancy $D_T^{(\mathcal{B})}$ of the continuous irrational rotation with slope $\alpha$ and starting point $\mathbf{x}$ with respect to the class $\mathcal{B}$ of axis-parallel rectangles in $I^2$ is 
 \begin{equation*}
  D_T^{(\mathcal{B})} = O(1)
 \end{equation*}
for all $\mathbf{x} \in I^2$ and almost all $\alpha >0$.
\end{corollary}

As clarified by the example in Figure \ref{fig:fullpgram}, an analogous result does not hold if $\mathcal{B}$ is the class of \emph{all} rectangles. It follows that the isotropic discrepancy, i.e.\ the discrepancy with respect to the class of all convex sets, cannot be bounded. However, if we let $\mathcal{B}$ be the class $\mathcal{D}$ of all discs in $I^2$, then we can attain a result analogous to Corollary \ref{cor:discr1}. Theorem \ref{thm:mainconvex} tells us that for all $S \in \mathcal{D}$, we have
\begin{equation*}
 \Delta_T(S, \alpha, \mathbf{x}) = O(1)
\end{equation*}
for all $\mathbf{x}$ and almost all $\alpha$, and from the proof of Theorem \ref{thm:mainconvex} it is not difficult to see that the $O$-constant can be made independent of the size and position of the disc $S$. We thus get:
\begin{corollary}
\label{cor:discr2}
The discrepancy $D_T^{(\mathcal{D})}$ of the continuous irrational rotation with slope $\alpha$ and starting point $\mathbf{x}$ with respect to the class $\mathcal{D}$ of discs in $I^2$ is
\begin{equation*}
 D_T^{(\mathcal{D})}=O(1)
\end{equation*}
for all $\mathbf{x} \in I^2$ and almost all $\alpha>0$.
\end{corollary}

The rest of the paper is organized as follows. In Section \ref{sec:proofs} we present necessary preliminary material, and give the proofs of Theorems \ref{thm:mainpoly} and \ref{thm:mainconvex}. Section \ref{sec:neg} is devoted to the proof of Theorem \ref{thm:mainneg}. 

%%%%%%%%%%%%%%%%%%%%%%%%%%%%%%%%%%%%%%%%% Proofs %%%%%%%%%%%%%%%%%%%%%%%%%%%%%%%%%%%%%%%%%

\section{Preliminaries and proofs of Theorems \ref{thm:mainpoly} and \ref{thm:mainconvex}}
\label{sec:proofs}
\subsection{Continued fractions}
We begin by briefly reviewing some well-known facts about continued fractions. For an irrational $\alpha \in (0,1)$, let 
\begin{equation*}
[0; a_1, a_2, a_3, \ldots ]
\end{equation*}
be its continued fraction expansion, and denote by $p_n / q_n$ its $n$th convergent. The numerators $p_n$ and denominators $q_n$ are given recursively by
\begin{equation*}
 \begin{aligned}
  q_0 = 1, \quad &q_1 = a_1 , \quad &q_{n+1} = a_{n+1} q_n + q_{n-1} , \\
  p_0 = 0, \quad &p_1 = 1 , \quad &p_{n+1} = a_{n+1} p_n + p_{n-1} . \\
 \end{aligned}
\end{equation*}
It follows readily from these recurrences that 
\begin{equation}
\label{eq:cfminusone}
 p_nq_{n+1} - p_{n+1} q_n = (-1)^{n+1}.
\end{equation}
The $n$th convergent $p_n/q_n$ is greater than $\alpha$ for every odd value of $n$, and smaller than $\alpha$ for every even value of $n$.
It is easy to see that $\lim_{n \rightarrow \infty} p_n / q_n = \alpha$, and moreover we have the error bounds
\begin{equation}
\label{eq:cferror}
 \frac{1}{(a_{n+1}+2)q_n^2} \leq \left| \alpha - \frac{p_n}{q_n} \right| \leq \frac{1}{a_{n+1}q_n^2} .
\end{equation}

Every non-negative integer $N$ has a unique expansion 
\begin{equation*}
 N = \sum_{i=0}^s b_i q_i, \text{ with } b_s>0 ; \quad 0 \leq b_i \leq a_{i+1} , \quad 0 \leq i \leq s .
\end{equation*}
We will refer to this as the \emph{Ostrowski expansion} of $N$ to base $\alpha$. 

Finally, we will need the following result, which follows from well-known facts in metric theory for continued fractions (see e.g.\ \cite{khinchin}).
\begin{lemma}
\label{lem:cfsum}
For almost every irrational $\alpha \in (0,1)$ and every $m>0$, the sum 
\begin{equation*}
 \sum_{l=0}^s \frac{a_{l+1}}{q_l^{1/m}} \sum_{k=1}^{l+1} a_k
\end{equation*}
is uniformly bounded in $s$.
\end{lemma}

\subsection{Functions of bounded remainder}
It is not difficult to show that the question of whether $S \subset I^2$ is a bounded remainder set for the continuous two-dimensional irrational rotation is essentially a one-dimensional problem. By making an appropriate projection, the question can be restated as that of whether a certain associated \emph{function} is of bounded remainder.
\begin{definition}
\label{def:brf}
 Let $f : \br \mapsto \bc$ be a $1$-periodic function which is integrable over $[0,1]$. We say that $f$ is a \emph{bounded remainder function} with respect to $\alpha \in \br \setminus \bq$ if there is a constant $C=C(f,\alpha)$ such that 
\begin{equation*}
 \left| \sum_{k=0}^{N-1} f(k\alpha) - N \int_0^1 f(x) \, dx \right| \leq C 
\end{equation*}
for all integers $N>0$.
\end{definition}
Bounded remainder functions have been studied by several authors, see e.g.\ \cite{hellekalek}, or \cite{schoissen} and the references therein.

We will consider two special classes of functions: \emph{hat functions} and \emph{dome functions}.
\begin{definition}
\label{def:hatfnc}
 We say that $T \,: \, \br \mapsto [0, \infty)$ is a \emph{hat function} if $T$ is supported on the interval $[0,b]$, with $b >0$, and 
\begin{equation}
\label{eq:hatfnc}
 T(x) = \begin{cases}
         \frac{H}{a} x \text{ ,} & 0 \leq x \leq a \text{ ;}\\
         -\frac{H}{b-a} (x-b) \text{ ,} & a < x \leq b \text{ ,}
        \end{cases}
\end{equation}
for some $0<a<b$ and $H>0$. 
\end{definition}

\begin{definition}
 \label{def:domefnc}
 We say that a continuous function $T \,: \, \br \mapsto [0, \infty)$ supported on $[0,B]$, with $B>0$, is a \emph{dome function} if it satisfies the following two conditions:
 \begin{enumerate} 
  \item \label{it:domediff} $T$ is concave and twice differentiable on the open interval $(0,B)$.
  \item \label{it:domegrowth} There exist $\varepsilon >0$, $m>0$ and $c>0$ such that
  \begin{equation}
  \label{eq:growthcond}
  \begin{aligned}
   &\left| T(x) \right| \leq c \cdot x^{1/m} \quad &\text{ for all } \, 0 \leq x < \varepsilon ; \\
   &\left| T(B-x)\right| \leq c \cdot x^{1/m} \quad &\text{ for all } \, 0 \leq x < \varepsilon .\\
  \end{aligned}
 \end{equation}
 \end{enumerate}
\end{definition}

We will establish and prove the following two results, which will be crucial for the proofs of Theorems \ref{thm:mainpoly} and \ref{thm:mainconvex} later on.
\begin{proposition}
 \label{prop:hatfncbrf}
 Let $\tau (x) = \sum_{m\in \bz} T(x+m)$, where $T$ is a hat function. Then $\tau$ is a bounded remainder function with respect to almost every $\alpha \in \br \setminus \bq$. 
\end{proposition}
\begin{proposition}
 \label{prop:domefncbrf}
 Let $\tau (x) = \sum_{m\in \bz} T(x+m)$, where $T$ is a dome function. Then $\tau$ is a bounded remainder function with respect to almost every $\alpha \in \br \setminus \bq$.
\end{proposition}
\begin{remark}
 \label{rem:shifts}
 For sufficiently regular functions, including periodizations of hat and dome functions, the bounded remainder property is not affected by shifting the function (see \cite[p.\ 128--129]{schoissen}). It thus follows from Propositions \ref{prop:hatfncbrf} and \ref{prop:domefncbrf} that for almost every $\alpha \in \br \setminus \bq$, we have
 \begin{equation*}
  \left| \sum_{k=0}^{N-1} \tau (k\alpha +x_0) - N \int_0^1 \tau (x) \right| \leq C
 \end{equation*}
for all $N>0$ and every $x_0 \in \br$ whenever $\tau$ is the periodization of a hat or dome function. The constant $C$ may depend on $\tau$ and $\alpha$, but not on $N$ or $x_0$.
\end{remark}
Later on we explain how Theorems \ref{thm:mainpoly} and \ref{thm:mainconvex} follow from the results above. 

For the proof of Proposition \ref{prop:hatfncbrf}, we will need the following lemma.
\begin{lemma}
\label{lem:altsum}
 Let $f \, : \, \br \mapsto \br$ be a $1$-periodic function, $\alpha$ be irrational and $N$ be a nonnegative integer with Ostrowski expansion 
 \begin{equation*}
 N= b_s q_s + \ldots + b_0 q_0
 \end{equation*}
to base $\alpha$. We then have
 \begin{equation}
 \label{eq:altsum}
 \sum_{k=0}^{N-1} f(k\alpha) = \sum_{l=0}^{s} \sum_{b=0}^{b_l-1} \sum_{k=0}^{q_l-1} f \left( \frac{k}{q_l} + \frac{\rho_{k,l}}{q_l} \right),
 \end{equation}
for some $\rho_{k,l}$ satisfying $-1<\rho_{k,l}<2$.
\end{lemma}
\begin{proof}
Let $n(0)=0$ and $n(l)=b_{l-1} q_{l-1} + \ldots +b_0q_0$ for $1\leq l\leq s$. It is straightforward to show that
\begin{equation}
\label{eq:ostrowskisum}
 \sum_{k=0}^{N-1} f(k\alpha) = \sum_{l=0}^{s} \sum_{b=0}^{b_l-1} \sum_{k=0}^{q_l-1} f(k\alpha +(n(l)+bq_l)\alpha).
\end{equation}
We define $\theta_l$ from the equation
\begin{equation*}
 \frac{\theta_l}{a_{l+1}q_l^2} = \alpha -\frac{p_l}{q_l} ,
\end{equation*}
and observe that by \eqref{eq:cferror} we have $1/3\leq|\theta_l|\leq1$. Moreover, we find $x_l\in [0,1)$ and $m_l\in \{ 0, \ldots , q_l-1 \}$, $m_l = m_l(b,x,\alpha)$, such that
\begin{equation*}
 \left\{(n(l)+bq_l)\alpha \right\} = \frac{m_l}{q_l} + \frac{x_l}{q_l}.
\end{equation*}
We can then rewrite the summand on the right hand side in \eqref{eq:ostrowskisum} as 
\begin{equation}
\label{eq:rewrittensummand}
 f(k\alpha +(n(l)+bq_l)\alpha) = f\left(\frac{k p_l+m_l}{q_l} + \frac{k\theta_l}{a_{l+1}q_l^2}+ \frac{x_l}{q_l} \right) .
\end{equation}
Using the substitution $kp_l + m_l = t \, (\modu q_l)$, which by \eqref{eq:cfminusone} gives 
\begin{equation*}
 k=(t-m_l)q_{l-1}(-1)^{l-1} \, (\modu q_l), 
\end{equation*}
we get 
\begin{equation}
\label{eq:rewrittenfill}
 \left\{ \frac{kp_l+m_l}{q_l} + \frac{k\theta_l}{a_{l+1}q_l^2}+\frac{x_l}{q_l} \right\} = \left\{ \frac{t}{q_l} + \frac{\rho_{t,l}}{q_l} \right\},
\end{equation}
where
\begin{equation}
 \label{eq:rho}
 \rho_{t,l} := \left\{ (t-m_l)(-1)^{l-1}\frac{q_{l-1}}{q_l} \right\} \frac{\theta_l}{a_{l+1}} + x_l .
\end{equation}
With this definition we have
\begin{equation*}
 - \frac{1}{a_{l+1}} < \rho_{t,l} < \frac{1}{a_{l+1}}+1 ,
\end{equation*}
and hence $-1 < \rho_{t,l} < 2$. Combining \eqref{eq:ostrowskisum}, \eqref{eq:rewrittensummand} and \eqref{eq:rewrittenfill}, we thus arrive at \eqref{eq:altsum}.
\end{proof}

\begin{proof}[Proof of Proposition \ref{prop:hatfncbrf}]
It will be sufficient to prove Proposition \ref{prop:hatfncbrf} for the case when $b \leq 1$ in Definition \ref{def:hatfnc}. To see this, observe that any general hat function $T$ can be written as a sum of shifted hat functions $T_i$ with support $[0,b]$, $b \leq 1$. Since any finite sum of bounded remainder functions is again a bounded remainder function, the general case follows from the special case $\tau(x) = \sum_{m \in \bz} T_i(x+m)$.
 
Our goal is to show that for almost every $\alpha \in \br \setminus \bq$, we can find a constant $C=C(\alpha, \tau)$ such that 
\begin{equation}
\label{eq:brhatgoal}
 \left| \sum_{k=0}^{N-1} \tau (k\alpha) - N \int_0^1 \tau (x) \, dx \right| \leq C 
\end{equation}
for every integer $N>0$. It will be enough to verify this for $\alpha \in (0,1)$, as the sum in \eqref{eq:brhatgoal} depends only on the fractional part of $\alpha$. By Lemma \ref{lem:altsum} we may rewrite this sum as 
\begin{equation*}
\sum_{l=0}^{s} \sum_{b=0}^{b_l-1} \sum_{k=0}^{q_l-1} \tau \left( \frac{k}{q_l} + \frac{\rho_{k,l}}{q_l} \right),
\end{equation*}
where $N=b_sq_s + \cdots + b_0q_0$ is the Ostrowski expansion of $N$ to base $\alpha$ and $-1 <\rho_{k,l} < 2$. We verify \eqref{eq:brhatgoal} in two steps: First we show that 
\begin{equation}
 \label{eq:hatstep1}
 \left| \sum_{l=0}^{s} \sum_{b=0}^{b_l-1} \sum_{k=0}^{q_l-1} \tau \left( \frac{k}{q_l} \right) - N \int_0^1 \tau (x) \, dx \right| \leq C , \quad N=1,2, \ldots ,
\end{equation}
for almost every irrational $\alpha \in (0,1)$. We then show that 
\begin{equation}
 \label{eq:hatstep2}
 \left| \sum_{l=0}^{s} \sum_{b=0}^{b_l-1} \sum_{k=0}^{q_l-1} \left( \tau \left( \frac{k}{q_l} + \frac{\rho_{k,l}}{q_l} \right)- \tau \left( \frac{k}{q_l}\right) \right) \right| \leq C , \quad s=1,2, \ldots , 
\end{equation}
for almost every irrational $\alpha \in (0,1)$. Combining \eqref{eq:hatstep1} and \eqref{eq:hatstep2}, we immediately obtain \eqref{eq:brhatgoal}.

Let us first verify that \eqref{eq:hatstep1} holds. On the interval $I$, the function $\tau$ is of the form \eqref{eq:hatfnc} with $b \leq 1$, so we can find $u_l, v_l \in \{ 0,1, \ldots , q_l-1 \}$ and $\xi_l , \eta_l \in (0,1]$ such that
\begin{equation}
\label{eq:aandb}
 a=\frac{u_l + \xi_l}{q_l} \quad \text{ and } \quad b=\frac{v_l+\eta_l}{q_l} .
\end{equation}
For sufficiently large $l>l_0$ (where $l_0=l_0(\tau)$ depends only on $\tau$), we have $u_l < v_l$, and a straightforward calculation gives
\begin{equation*}
 \sum_{k=0}^{q_l-1} \tau \left( \frac{k}{q_l} \right) = \frac{Hb}{2} q_l + \frac{Ha\eta_l(1-\eta_l)-Hb\xi_l(1-\xi_l)}{2a(b-a)q_l}.
\end{equation*}
Thus we have  
\begin{equation*}
\left| \sum_{k=0}^{q_l-1} \tau \left( \frac{k}{q_l} \right) - q_l \int_0^1 \tau(x) \, dx \right| \leq C \frac{1}{q_l},
\end{equation*}
where $C=C(\tau)$ (this is trivially true also when $l \leq l_0$), and it follows that
\begin{equation*}
\left| \sum_{l=0}^{s} \sum_{b=0}^{b_l-1} \sum_{k=0}^{q_l-1} \tau \left( \frac{k}{q_l} \right) - N \int_0^1 \tau (x) \, dx \right| \leq C \sum_{l=0}^s \frac{b_l}{q_l} .
\end{equation*}
Since $b_l < a_{l+1}$, it follows from Lemma \ref{lem:cfsum} that the right hand side above is uniformly bounded in $s$ for almost every $\alpha \in (0,1)$. This confirms \eqref{eq:hatstep1}.

We go on to verify \eqref{eq:hatstep2}. We will assume below that $b < 1$ in \eqref{eq:hatfnc}; the proof when $b=1$ is slightly simpler, but essentially the same. Let $Q_l = \{ 0,1, \ldots , q_l-1\}$, and define $u_l, v_l \in Q_l$ as in \eqref{eq:aandb}. Denote by $E$ a set of ``exceptional'' indices
\begin{equation*}
E = \{ 0,u_l-1, u_l, u_l+1, v_l-1, v_l, v_l+1, q_l-1\}
\end{equation*}
(for sufficiently large $l \geq l_0$, these are all distinct). We have
\begin{equation}
\label{eq:indexsplit}
\sum_{k=0}^{q_l -1} \tau \left( \frac{k}{q_l} + \frac{\rho_{k,l}}{q_l} \right) = \sum_{k \in Q_l \setminus E}  \tau \left( \frac{k}{q_l} + \frac{\rho_{k,l}}{q_l} \right) + \sum_{k \in E}  \tau \left( \frac{k}{q_l} + \frac{\rho_{k,l}}{q_l} \right),
\end{equation}
and since $\tau$ is everywhere linear (with bounded slope) it is clear that 
\begin{equation}
\label{eq:indexsplit1}
\sum_{k \in E} \tau \left( \frac{k}{q_l} + \frac{\rho_{k,l}}{q_l} \right) = \sum_{k \in E} \tau \left( \frac{k}{q_l} \right) + O \left( \frac{1}{q_l} \right).
\end{equation}
The second sum on the right hand side in \eqref{eq:indexsplit} can be rewritten using the specific form \eqref{eq:hatfnc} of $\tau$ on $I$. We get
\begin{equation}
\label{eq:indexsplit2}
\sum_{k \in Q_l \setminus E}  \tau \left( \frac{k}{q_l} + \frac{\rho_{k,l}}{q_l} \right) = \sum_{k \in Q_l \setminus E}  \tau \left( \frac{k}{q_l} \right) + \Sigma_1 ,
\end{equation}
where 
\begin{equation*}
\Sigma_1 := \frac{1}{q_l} \left( \frac{H}{a} \sum_{k=1}^{u_l-2} \rho_{k,l} - \frac{H}{b-a} \sum_{k=u_l+2}^{v_l-2} \rho_{k,l} \right) ,
\end{equation*}
and $\rho_{k,l}$ is defined in \eqref{eq:rho}.
To verify \eqref{eq:hatstep2}, we will need to find an appropriate bound on $\Sigma_1$. 

We now show that $\Sigma_1 = O( \sum_{i=1}^l a_i /q_l)$. By defining $\alpha_l$ and $\gamma_l$ as
\begin{equation}
\label{eq:alphgam}
\begin{aligned}
  \alpha_l &:= (-1)^{l-1} \frac{q_{l-1}}{q_l} , \\
  \gamma_l &:= -m_l(-1)^{l-1}\frac{q_{l-1}}{q_l} ,
 \end{aligned}
\end{equation}
we can rewrite $\rho_{k,l}$ in \eqref{eq:rho} as
\begin{equation*}
 \rho_{k,l} = \omega_{k,l} \cdot \frac{\theta_l}{a_{l+1}} + x_l ,
\end{equation*}
where $\omega_{k,l} := \{ k\alpha_l + \gamma_l \}$. Using \eqref{eq:aandb} and the fact that $x_l \in [0,1)$, it is an easy task to show that
\begin{equation*}
 \frac{H}{a} \sum_{k=1}^{u_l-2} x_l - \frac{H}{b-a} \sum_{k=u_l+2}^{v_l-2} x_l = O(1) .
\end{equation*}
We thus have
\begin{equation}
\label{eq:S1calc1}
\begin{aligned}
 \Sigma_1 &= \frac{\theta_l}{q_l a_{l+1}} \left( \frac{H}{a} \sum_{k=1}^{u_l-2} \omega_{k,l} - \frac{H}{b-a} \sum_{k=u_l+2}^{v_l-2} \omega_{k,l} \right) + O \left( \frac{1}{q_l} \right) \\
&= \frac{H\theta_l}{q_l a_{l+1}} \left( \frac{1}{a} \sum_{k=0}^{u_l-1} \omega_{k,l} - \frac{1}{b-a} \sum_{k=u_l}^{v_l-1} \omega_{k,l} \right) + O \left( \frac{1}{q_l} \right) ,\\
 \end{aligned}
 \end{equation}
where the last equality follows from boundedness of the terms $\omega_{k,l}$. 

To further approximate $\Sigma_1$, we employ Koksma's inequality for the sequence $\{\omega_{k,l} \}_{k=0}^{q_l-1}$ and the linear function $f(x) = \{ x\}$ (see \cite[Theorem 5.1]{kuipers}). For $1 \leq N \leq q_l$, we have
\begin{equation*}
 \left| \sum_{k=0}^{N-1} \omega_{k,l} - N \int_0^1 x \, dx \right| =  \left| \sum_{k=0}^{N-1} \omega_{k,l} - \frac{N}{2} \right| \leq ND_N^*(\omega_{k,l}) V_I(f),
\end{equation*}
where $V_I(f) =1$ is the total variation of $f$ over $I$, and $D_N^*(\omega_{k,l})$ denotes the star-discrepancy of the point set $\{\omega_{k,l} \}_{k=0}^{N-1}$.
The extreme discrepancy $D_N$ of $\{ \omega_{k,l} \}_{k=0}^{N-1}$ equals that of $\{ k\alpha_l\}_{k=0}^{N-1}$. Note that $|\alpha_l| = q_{l-1}/q_l$ has continued fraction expansion 
\begin{equation*}
 |\alpha_l | = \left[ 0; a_l, a_{l-1}, \ldots a_1 \right].
\end{equation*}
It thus follows that 
\begin{equation}
\label{eq:boundomega2}
 ND_N^* (\omega_{k,l}) \leq ND_N(\omega_{k,l}) = ND_N(k\alpha_l) \leq 1 + 2 \sum_{i=1}^l a_i
\end{equation}
for $1 \leq N \leq q_l$ (see \cite[p.\ 126]{kuipers} for the last inequality). Hence, we have 
\begin{equation*}
 \left| \sum_{k=0}^{N-1} \omega_{k,l} - \frac{N}{2} \right| \leq 1 + 2 \sum_{i=1}^l a_i ,
\end{equation*}
and from this and \eqref{eq:S1calc1} it follows that
\begin{equation*}
 \begin{aligned}
  \Sigma_1 &= \frac{H \theta_l}{q_l a_{l+1}} \left( \left( \frac{1}{a} + \frac{1}{b-a} \right) \sum_{k=0}^{u_l-1} \omega_{k,l} - \frac{1}{b-a} \sum_{k=0}^{v_l-1} \omega_{k,l} \right) + O \left( \frac{1}{q_l} \right) \\
  &= \frac{H\theta_l}{q_l a_{l+1}} \left( \frac{b}{a(b-a)} \cdot \frac{u_l}{2} - \frac{1}{b-a} \cdot \frac{v_l}{2} \right) + O\left( \frac{\sum_{i=1}^l a_i}{q_l} \right) \\
  &= \frac{H\theta_l}{2q_l a_{l+1}} \left( \frac{b}{a(b-a)} \left( q_l a-\xi_l\right) - \frac{1}{b-a} \left( q_lb-\eta_l\right) \right) + O\left( \frac{\sum_{i=1}^l a_i}{q_l} \right) \\
  &= O\left( \frac{\sum_{i=1}^l a_i}{q_l} \right) .
 \end{aligned}
\end{equation*}

Let us finally see that this bound on $\Sigma_1$ implies \eqref{eq:hatstep2}. Inserting \eqref{eq:indexsplit1} and \eqref{eq:indexsplit2} in \eqref{eq:indexsplit}, we get
\begin{equation*}
 \left| \sum_{k=0}^{q_l-1} \tau \left( \frac{k}{q_l} + \frac{\rho_{k,l}}{q_l} \right) - \sum_{k=0}^{q_l-1} \tau \left( \frac{k}{q_l} \right)\right| \leq \frac{C}{q_l} \sum_{i=1}^l a_i ,
\end{equation*}
for $l \geq l_0=l_0(\tau)$ and some constant $C$ which depends only on $\tau$ and $\alpha$ (this bound holds trivially also when $l < l_0$). We thus have
\begin{equation*}
\begin{aligned}
  \left| \sum_{l=0}^{s} \sum_{b=0}^{b_l-1} \sum_{k=0}^{q_l-1} \left( \tau \left( \frac{k}{q_l} + \frac{\rho_{k,l}}{q_l} \right)- \tau \left( \frac{k}{q_l}\right) \right) \right| &\leq C' a_1 + C \sum_{l=1}^s \frac{b_l}{q_l} \sum_{i=1}^l a_i \\
  &\leq C \sum_{l=0}^s \frac{a_{l+1}}{q_l} \sum_{i=1}^{l+1} a_i.
\end{aligned}
\end{equation*}
By Lemma \ref{lem:cfsum}, the sum on the right hand side above is bounded uniformly in $s$ for almost every irrational $\alpha \in (0,1)$. This verifies \eqref{eq:hatstep2}, and completes the proof of Proposition \ref{prop:hatfncbrf}.
\end{proof}

Before we embark on the proof of Proposition \ref{prop:domefncbrf}, we establish the following preliminary result.
\begin{lemma}
 \label{lem:boundedvar}
 Suppose $f : \br \mapsto \br$ is a dome function as given in Definition \ref{def:domefnc}, and let $q>2/B$. 
 Denote by $f_q'$ the function 
 \begin{equation}
 \label{eq:derivq}
  f_q'(x) = \begin{cases}
          f'(x) & 1/q \leq x \leq B-1/q , \\
          0 & \text{ otherwise. }
         \end{cases}
 \end{equation}
 Then for $q>1/\varepsilon$, with $\varepsilon$ as in \eqref{eq:growthcond}, the total variation $V_I(f_q')$ of $f_q'$ over $I$ satisfies
 \begin{equation*}
  V_I(f_q') \leq Cq^{1-1/m} ,
 \end{equation*}
where $C=C(c)$ with $c$ as in \eqref{eq:growthcond}.
\end{lemma}
\begin{proof}
 The function $f$ is concave and twice differentiable on $(0,B)$, from which it follows that $f'$ is monotonically nonincreasing and
 \begin{equation*}
  V_I(f_q') = 2 \left( f'\left( \frac{1}{q} \right) - f'\left( B-\frac{1}{q} \right) \right).
 \end{equation*}
Moreover, we have that 
\begin{equation*}
 f' \left( \frac{1}{q} \right) \leq \frac{f\left( 1/q\right)-f(0)}{1/q} = qf\left(\frac{1}{q}\right) , 
\end{equation*}
and likewise 
\begin{equation*}
f'\left(B-\frac{1}{q} \right) \geq -qf\left(B-\frac{1}{q} \right).
\end{equation*}
By the conditions \eqref{eq:growthcond} on $f$ it thus follows that 
\begin{equation*}
 V_I(f_q') \leq 4cq^{1-1/m} \quad \text{ for all } q> \frac{1}{\varepsilon} .
\end{equation*}
\end{proof}

\begin{proof}[Proof of Proposition \ref{prop:domefncbrf}]
It will be sufficient to prove Proposition \ref{prop:domefncbrf} for the case when $B \leq 1$ in Definition \ref{def:domefnc}. To see this, observe that any general dome function $T$ can be written as a sum of shifted hat functions, and shifted dome functions with support in $I$. This is illustrated for the case $1 < B \leq 2$ in Figure \ref{fig:domedecomp}; we may write the function $T$ as
\begin{equation*}
 T= T_1 + T_2 + T_3 ,
\end{equation*}
where $T_1$ is the hat function in \eqref{eq:hatfnc} with $a=1, b=B$ and $H=T(1)$, and $T_2$ and $T_3$ are the dome functions $T_2 = \chi_{[0,1]}\cdot (T-T_1)$ and $T_3=\chi_{[1,B]}\cdot (T-T_1)$. As the sum of finitely many bounded remainder functions is again a bounded remainder function, the general case follows from the special case $T=T_3$ and Proposition \ref{prop:hatfncbrf}.
In other words, it is sufficient to consider the case when, restricted to the unit interval, $\tau$ is simply a dome function with support $[0,B]$, $B \leq 1$. 
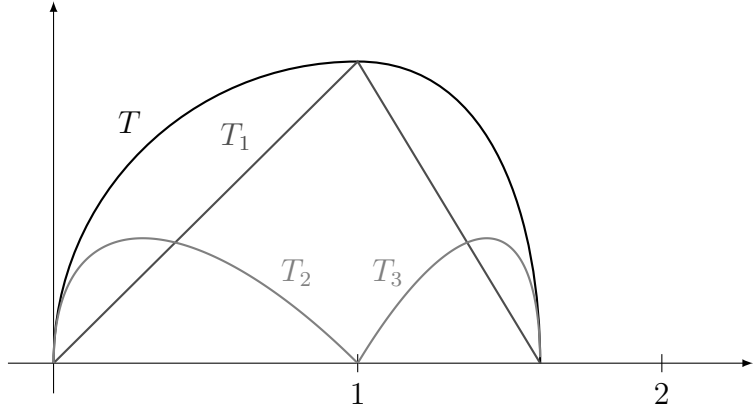
\begin{figure}
 \centering
 \begin{tikzpicture}[scale=4]
%axes
\draw[-latex] (-.15,0) -- (2.3, 0);
\draw[-latex] (0,-.1) -- (0, 1.2);
\draw (1,-.03) -- (1, .03);
\draw(1,-.03) node[below]{$1$};
\draw (2,-.03) -- (2, .03);
\draw(2,-.03) node[below]{$2$};
% graph of T and decomp
\draw[black, thick] (0,0) to [out=90, in=180] (1,1) to [out=0, in=90] (1.6,0);
\draw(.25, .8) node[black]{$T$};
\draw[{black!70!white}, thick] (0,0) -- (1,1) -- (1.6,0);
\draw(.6,.75) node[{black!70!white}]{$T_1$};
\draw[{black!50!white}, thick, domain=0:1, samples=500] plot (\x, {sqrt(1-(\x-1)^2)-\x});
\draw(.8,.3) node[{black!50!white}]{$T_2$};
\draw[{black!50!white}, thick, domain=1:1.6, samples=500] plot (\x, {sqrt(1-((\x-1)/.6)^2)+(5/3)*(\x-1.6)});
\draw(1.1,.3) node[{black!50!white}]{$T_3$};
\end{tikzpicture}
 \caption{The dome function $T$ decomposed as the sum of a hat function $T_1$, and two dome functions $T_2$ and $T_3$ supported on intervals of length at most one.\label{fig:domedecomp}}
\end{figure}

Let $\tau$ be such a function. We want to show that for almost every $\alpha \in \br \setminus \bq$, we can find a constant $C=C(c,m,\alpha)$ such that 
 \begin{equation}
  \label{eq:brfbowler}
   \left| \sum_{k=0}^{N-1} \tau(k \alpha) - N \int_0^1 \tau(x) \, dx \right| \leq C 
 \end{equation}
 for every integer $N>0$. Again it will be enough to verify this for $\alpha \in (0,1)$, as the sum in \eqref{eq:brfbowler} depends only on the fractional part of $\alpha$. By Lemma \ref{lem:altsum}, we may rewrite this sum as 
\begin{equation*}
 \sum_{l=0}^{s} \sum_{b=0}^{b_l-1} \sum_{k=0}^{q_l-1} \tau \left( \frac{k}{q_l} + \frac{\rho_{k,l}}{q_l} \right),
\end{equation*}
where $N=b_sq_s + \cdots + b_0q_0$ is the Ostrowski expansion of $N$ to base $\alpha$ and $-1<\rho_{k,l}<2$. We verify \eqref{eq:brfbowler} in two steps:  First we show that 
\begin{equation}
 \label{eq:bowler1}
 \left| \sum_{l=0}^{s} \sum_{b=0}^{b_l-1} \sum_{k=0}^{q_l-1} \tau \left( \frac{k}{q_l} \right) - N \int_0^1 \tau (x) \, dx \right| \leq C , \quad N=1,2, \ldots ,
\end{equation}
for almost every irrational $\alpha \in (0,1)$. We then show that 
\begin{equation}
 \label{eq:bowler2}
 \left| \sum_{l=0}^{s} \sum_{b=0}^{b_l-1} \sum_{k=0}^{q_l-1} \left( \tau \left( \frac{k}{q_l} + \frac{\rho_{k,l}}{q_l} \right)- \tau \left( \frac{k}{q_l}\right) \right) \right| \leq C , \quad s=1,2, \ldots , 
\end{equation}
for almost every irrational $\alpha \in (0,1)$. 
Combining \eqref{eq:bowler1} and \eqref{eq:bowler2}, we immediately obtain \eqref{eq:brfbowler}.

Let us first see that \eqref{eq:bowler1} holds. On the interval $I$, the function $\tau$ is supported on $[0,B]$ with $0<B\leq1$, so
we can find $u_l \in \{ 0, 1, \ldots , q_l-1\}$ and $\xi_l \in (0,1]$ such that 
\begin{equation}
\label{eq:ulxil}
B = \frac{u_l + \xi_l}{q_l}.
\end{equation}
We begin by considering the inner sum
\begin{equation*}
\sum_{k=0}^{q_l-1} \tau \left( \frac{k}{q_l} \right) = \sum_{k=1}^{u_l-1} \tau \left( \frac{k}{q_l} \right) + \tau \left( \frac{u_l}{q_l} \right) .
\end{equation*}
It is not difficult to show, for instance using integration by parts, that 
\begin{equation*}
\begin{aligned}
\sum_{k=1}^{u_l-1} \tau \left( \frac{k}{q_l} \right) &= q_l \int_{1/q_l}^{(u_l-1)/q_l} \tau(x) \, dx + \frac{1}{2} \left( \tau \left( \frac{1}{q_l}\right)+ \tau \left( \frac{u_l-1}{q_l}\right) \right)\\
&+ \int_{1/q_l}^{(u_l-1)/q_l} \left( \{ q_lx \} - \frac{1}{2} \right) \tau'(x) \, dx , 
\end{aligned}
\end{equation*}
and hence 
\begin{equation*}
\begin{aligned}
\sum_{k=0}^{q_l-1} \tau \left( \frac{k}{q_l} \right) - q_l\int_0^1 \tau(x) \, dx &= \tau \left( \frac{u_l}{q_l} \right) + \frac{1}{2} \left( \tau \left( \frac{1}{q_l}\right) + \tau \left( \frac{u_l-1}{q_l} \right)\right) \\
&- q_l \left( \int_0^{1/q_l} \tau(x) \, dx + \int_{(u_l-1)/q_l}^B \tau(x) \, dx \right)\\
&+  \int_{1/q_l}^{(u_l-1)/q_l} \left( \{ q_lx\} - \frac{1}{2} \right) \tau'(x) \, dx .
\end{aligned}
\end{equation*}
Now let $l > l_0=l_0(\tau)$ be sufficiently large for $q_l >2/\varepsilon$.  
It is then clear from the conditions \eqref{eq:growthcond} on $\tau$ that all but the last term on the right hand side above are bounded by $C q_l^{-1/m}$ in absolute value (where $C=C(c,m)$). In fact, the same bound holds also for the last term, as 
\begin{equation*}
\begin{aligned}
&\left| \int_{1/q_l}^{(u_l-1)/q_l} \left( \{q_l x\} - \frac{1}{2} \right) \tau'(x) \, dx \right| \\
&\leq \sum_{i=1}^{u_l-2} \left| \int_{i/q_l}^{(i+1)/q_l} \left( \{q_l x\} - \frac{1}{2} \right) \tau'(x) \, dx \right| \\
&\leq \sum_{i=1}^{u_l-2} \Big| \max_{x \in [\frac{i}{q_l}, \frac{i+1}{q_l}]} \tau'(x) - \min_{x \in [\frac{i}{q_l}, \frac{i+1}{q_l}]} \tau'(x) \Big| \int_{(2i+1)/2q_l}^{(i+1)/q_l} \left( \{ q_lx\} - \frac{1}{2} \right) \, dx \\
&\leq \frac{1}{8q_l} V_I(\tau'_{q_l}) ,
\end{aligned}
\end{equation*}
with $\tau'_{q_l}$ defined as in \eqref{eq:derivq}. Since $q_l> 2/\varepsilon$, it follows from Lemma \ref{lem:boundedvar} that \begin{equation*}
\frac{1}{8q_l} V_I(\tau'_{q_l}) \leq C q_l^{-1/m},
\end{equation*}
where $C=C(c)$, and hence we get
\begin{equation*}
\left| \sum_{k=0}^{q_l-1} \tau \left( \frac{k}{q_l} \right) - q_l\int_0^1 \tau(x) \, dx \right| \leq C q_l^{-1/m}, 
\end{equation*} 
for some constant $C(c, m)$ and $l>l_0$ (and this bound holds trivially also when $l \leq l_0$). 
It follows that 
\begin{equation*}
 \left| \sum_{l=0}^{s} \sum_{b=0}^{b_l-1} \sum_{k=0}^{q_l-1} \tau \left( \frac{k}{q_l} \right) - N \int_0^1 \tau (x) \, dx \right| \leq C \sum_{l=0}^s \frac{b_l}{q_l^{1/m}},
\end{equation*}
and by Lemma \ref{lem:cfsum} the latter sum is uniformly bounded in $s$ for almost every irrational $\alpha \in (0,1)$. This confirms \eqref{eq:bowler1}.

We now show that \eqref{eq:bowler2} holds. We assume below that $B < 1$; the proof when $B=1$ is slightly simpler, but essentially the same. Again we begin by treating the inner sum
\begin{equation}
\label{eq:b2inner}
\sum_{k=0}^{q_l-1} \left( \tau \left( \frac{k}{q_l} + \frac{\rho_{k,l}}{q_l} \right) - \tau \left( \frac{k}{q_l} \right) \right),
\end{equation}
which we will show is bounded in absolute value by
\begin{equation}
\label{eq:wantedbound}
\sum_{i=1}^l a_i \left( C_1q_l^{-1} + C_2q_l^{-1/m} \right),  
\end{equation}
for constants $C_1=C_1(m,c,\alpha)$ and $C_2=C_2(m,c,\alpha)$. 

Let $u_l$ be defined as in \eqref{eq:ulxil}, and denote by $E$ a set of ``exceptional'' indices
\begin{equation*}
E= \left\{ 0,1,u_l-2, u_l-1, u_l, u_l+1, q_l-1 \right\} 
\end{equation*} 
(for sufficiently large $l$, these are all distinct). We split the sum \eqref{eq:b2inner} into two parts
\begin{equation*}
\Sigma_1 := \sum_{k \in E} \left( \tau \left( \frac{k}{q_l} + \frac{\rho_{k,l}}{q_l} \right) - \tau \left( \frac{k}{q_l} \right) \right)
\end{equation*}
and
\begin{equation*}
\Sigma_2 := \sum_{k=2}^{u_l-3} \left( \tau \left( \frac{k}{q_l} + \frac{\rho_{k,l}}{q_l} \right) - \tau \left( \frac{k}{q_l} \right) \right) .
\end{equation*}
Now let $l>l_1$ be sufficiently large for $q_l > 4/\varepsilon$. Since $-1 < \rho_{k,l} < 2$, it follows from the conditions \eqref{eq:growthcond} on $\tau$ that 
\begin{equation}
\label{eq:S1bound}
\left| \Sigma_1 \right| \leq C q_l^{-1/m} ,
\end{equation}
where $C=C(c)$. To find a bound on $\Sigma_2$, we first rewrite the sum using the mean value theorem. We have that 
\begin{equation*}
\Sigma_2 = \sum_{k=2}^{u_l-3} \tau'(r_k) \frac{\rho_{k,l}}{q_l} ,
\end{equation*}
where $r_k \in (k/q_l, (k+\rho_{k,l})/q_l)$ if $\rho_{k,l}>0$ and $r_k \in ((k+\rho_{k,l})/q_l, k/q_l)$ if $\rho_{k,l}<0$. It follows that 
\begin{equation}
\label{eq:S2ineq}
\begin{aligned}
\left| \Sigma_2 - \sum_{k=2}^{u_l-3} \tau' \left(\frac{k}{q_l} \right) \frac{\rho_{k,l}}{q_l} \right| &= \left| \sum_{k=2}^{u_l-3} \left( \tau'(r_k) - \tau' \left(\frac{k}{q_l} \right)\right) \frac{\rho_{k,l}}{q_l} \right| \\
&\leq \frac{2}{q_l} \sum_{k=2}^{u_l-3} \max_{x,y \in \left[ \frac{k-1}{q_l}, \frac{k+2}{q_l} \right]} \left| \tau'(x) - \tau'(y) \right| \\
&\leq \frac{6}{q_l} V_I(\tau'_{q_l}) \leq Cq_l^{-1/m},
\end{aligned}
\end{equation}
where $C=C(c)$, and for the last inequality we have used Lemma \ref{lem:boundedvar}.

Finally, we need to find a bound on
\begin{equation*}
\sum_{k=2}^{u_l-3} \tau' \left(\frac{k}{q_l} \right) \frac{\rho_{k,l}}{q_l} .
\end{equation*}
Recall from the proof of Proposition \ref{prop:hatfncbrf} that we may write $\rho_{k,l}$ as
\begin{equation*}
 \rho_{k,l} = \omega_{k,l} \frac{\theta_l}{a_{l+1}}+x_l ,
\end{equation*}
where $\omega_{k,l} = \{k\alpha_l+\gamma_l\}$, and $\alpha_l$ and $\gamma_l$ are given in \eqref{eq:alphgam}. Let us define the two-dimensional sequence $\omega := (\omega_1(k), \omega_2(k))_{k=0}^{q_l-1}$, where
\begin{equation*}
 \omega_1(k) = \frac{k}{q_l} , \quad \omega_2(k) = \omega_{k,l} .
\end{equation*}
Moreover, let $G \, : \, I^2 \mapsto \br$ be the function given by
\begin{equation*}
 G(x,y) := \chi_{[2/q_l, (u_l-3)/q_l]}(x) \tau'(x) \cdot h(y),
\end{equation*}
where $h \, : \, I \mapsto \br$ is the linear function
\begin{equation*}
 h(y) := \frac{\theta_l}{a_{l+1}}y + x_l .
\end{equation*}
We then have
\begin{equation}
\label{eq:Gsum}
 \sum_{k=2}^{u_l-3} \tau' \left(\frac{k}{q_l} \right) \frac{\rho_{k,l}}{q_l} 
= \frac{1}{q_l} \sum_{k=0}^{q_l-1} G \left( \omega_1 (k), \omega_2(k) \right).
\end{equation}
From the two-dimensional Koksma-Hlawka inequality \cite[p.\ 151, p.\ 100]{kuipers} we get
\begin{equation}
\label{eq:KH2dim}
\begin{aligned}
&\left| \frac{1}{q_l} \sum_{k=0}^{q_l-1} G(w_1(k), w_2(k)) - \int_0^1 \int_0^1 G(x,y) \, dx dy \right| \\
&\leq D_{q_l}^*(\omega_1) V_I(\chi_{[2/q_l, (u_l-3)/q_l]} \tau') + D_{q_l}^*(\omega_2)V_I(h)+D_{q_l}^*(\omega)V_{I^2}(G) \\
&\leq D_{q_l}^*(\omega) \left( V_I(\chi_{[2/q_l, (u_l-3)/q_l]} \tau') + V_I(h) + V_{I^2}(G) \right) .
\end{aligned}
\end{equation}
We now use this inequality to find a bound on the sum \eqref{eq:Gsum}. It is not difficult (see e.g.\ \cite[p.\ 106]{kuipers}) to show that
\begin{equation}
\label{eq:qDq}
q_l D_{q_l}^*(\omega) \leq 2 q_l D_{q_l}^*(\omega_2) \leq 2 \left( 1+2\sum_{i=1}^l a_i \right),
\end{equation}
where for the second inequality we have used \eqref{eq:boundomega2}. Moreover, we have
\begin{equation}
\label{eq:Vh}
 V_I(h) = \frac{|\theta_l|}{a_{l+1}} \leq 1,
\end{equation}
and using monotonicity of $\tau'$ and Lemma \ref{lem:boundedvar} we get
\begin{equation}
 \label{eq:Vtau}
 V_I(\chi_{[2/q_l, (u_l-3)/q_l]} \tau') \leq V_I(\tau_{q_l}') \leq C q_l^{1-1/m}, 
\end{equation}
where $C=C(c)$ with $c$ as in \eqref{eq:growthcond}. It follows that
\begin{equation}
\label{eq:VG}
 V_{I^2}(G) \leq  V_I(\chi_{[2/q_l, (u_l-3)/q_l]} \tau') \cdot V_I(h) \leq C q_l^{1-1/m}.
\end{equation}
Lastly, we have that
\begin{equation*}
\left| \int_0^1 \int_0^1 G(x,y) \, dx dy \right| = \left| \left( \tau \left( \frac{u_l-3}{q_l} \right) - \tau \left( \frac{2}{q_l}\right)\right) \cdot \left( \frac{\theta_l}{2a_{l+1}}+x_l \right)\right|,
\end{equation*}
which by \eqref{eq:growthcond} is bounded by $Cq_l^{-1/m}$, $C=C(c,m)$, when $l>l_1$. 
Inserting \eqref{eq:qDq} -- \eqref{eq:VG} and this integral estimate in \eqref{eq:KH2dim}, we get
\begin{equation*}
 \begin{aligned}
  \left|  \frac{1}{q_l} \sum_{k=0}^{q_l-1} G \left( \omega_1 (k), \omega_2(k) \right) \right| &\leq Cq_l^{-1/m} + \frac{2}{q_l}\Big( 1+2 \sum_{i=1}^l a_i \Big)\left(1+2Cq_l^{1-1/m} \right) \\
  &\leq \sum_{i=1}^l a_i \left( C_1q_l^{-1} + C_2 q_l^{-1/m} \right),
 \end{aligned}
\end{equation*}
where the constants $C_1$ and $C_2$ depend only on $c$ and $m$ in \eqref{eq:growthcond}. It thus follows from \eqref{eq:Gsum} and \eqref{eq:S2ineq} that $|\Sigma_2|$ satisfies the bound \eqref{eq:wantedbound} for $l>l_1$. The same is true for $|\Sigma_1|$ by \eqref{eq:S1bound}, and hence $\Sigma_1 +\Sigma_2$ in \eqref{eq:b2inner} obeys the bound \eqref{eq:wantedbound} as well. 
We get 
 \begin{equation*}
  \begin{aligned}
   &\left| \sum_{l=0}^{s} \sum_{b=0}^{b_l-1} \sum_{k=0}^{q_l-1} \left( \tau \left( \frac{k}{q_l} + \frac{\rho_{k,l}}{q_l} \right)- \tau \left( \frac{k}{q_l}\right) \right) \right| \\
   &\leq C' a_1 + C_1 \sum_{l=1}^s \frac{b_l}{q_l} \sum_{i=1}^l a_i + C_2 \sum_{l=1}^s \frac{b_l}{q_l^{1/m}} \sum_{i=1}^l a_i \\
   &\leq C_1 \sum_{l=0}^s \frac{a_{l+1}}{q_l} \sum_{i=1}^{l+1} a_i + C_2 \sum_{l=0}^s \frac{a_{l+1}}{q_l^{1/m}} \sum_{i=1}^{l+1} a_i ,
  \end{aligned}
 \end{equation*}
 and from Lemma \ref{lem:cfsum} it follows that the latter expression is bounded uniformly in $s$ for almost every irrational $\alpha \in (0,1)$. This verifies \eqref{eq:bowler2}, and completes the proof of Proposition \ref{prop:domefncbrf}.
\end{proof}

\subsection{Proof of Theorems \ref{thm:mainpoly} and \ref{thm:mainconvex}}
We now turn to the proofs of Theorems \ref{thm:mainpoly} and \ref{thm:mainconvex}. We will begin by proving a lemma showing that the question of whether $S \subset I^2$ is a bounded remainder set can be restated as a question of whether an associated function is of bounded remainder.

Let $S \subset I^2$ be either a polygon or a set satisfying the conditions in Theorem \ref{thm:mainconvex}. We can then associate to $S$ a function $\tau_S : [0,1) \mapsto [0,\infty)$ defined as
\begin{equation}
\label{eq:tau}
 \tau_S(x) := \int_0^1 \chi_S(t, \{t\alpha+x\}) \, dt .
\end{equation}
A geometric interpretation of $\tau_S$ is illustrated in Figure \ref{fig:tauS}. It is easy to show that 
\begin{equation*}
 \int_0^1 \tau_S(x) \, dx = \lambda (S) .
\end{equation*}
Moreover, we have the following:
\begin{lemma}
\label{lem:equivs}
The set $S\subset I^2$ is a bounded remainder set for the irrational rotation with slope $\alpha>0$ and starting point $\vx = (x_1, x_2) \in I^2$ if and only if $\tau_S$ is a bounded remainder function with respect to $\alpha$.
\end{lemma}
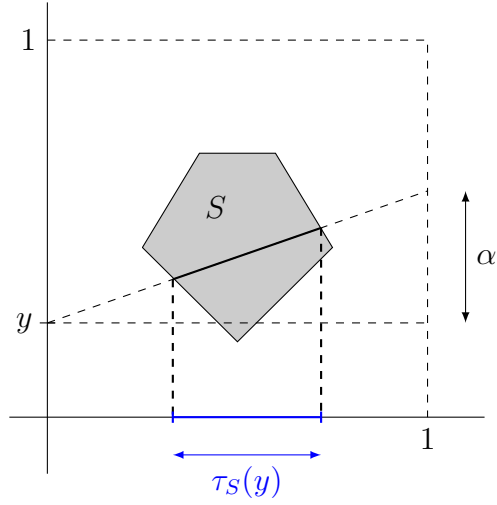
\begin{figure}
 \centering
 \begin{tikzpicture}[scale=5]
  %axes
\draw (-.1,0) -- (1.15, 0);
\draw (0,-.15) -- (0, 1.1);
%limits of unit square
\draw[dashed] (0,1) -- (1,1) -- (1,0);
\draw(0,1) node[left]{$1$};
\draw(1,0) node[below]{$1$};
%the set S
\draw[draw=black,fill={black!20!white}] (.25,.45) -- (.5,.2) -- (.75,.45) -- (.6,.7) -- (.4,.7) -- (.25,.45);
\draw(.5,.5) node[above left]{$S$};
%lines related to tau
\draw (-.02,.25) -- (.02,.25);
\draw (-.01, .25) node[left]{$y$};
\draw[dashed] (0,.25) -- (1, 0.6);
\draw[thick] (.33, {.25+.35*.33}) -- (.72, {.25+.35*.72});
\draw[dashed, thick] (.33, {.25+.35*.33}) -- (.33,0);
\draw[dashed, thick] (.72, {.25+.35*.72}) -- (.72,0);
\draw[dashed] (0,.25) -- (1,.25);
\draw[blue, thick] (.33,0) -- (.72, 0);
\draw[blue, thick] (.33,.015) -- (.33,-.015);
\draw[blue, thick] (.72, .015) -- (.72, -.015);
%arrows
\draw[latex-latex, blue] (.33,-.1) -- (.72,-.1);
\draw[blue] ({(.33+.72)/2}, -.1) node[below]{$\tau_S(y)$};
\draw[latex-latex] (1.1,.25) -- (1.1, .6);
\draw (1.1,{(.25+.6)/2}) node[right]{$\alpha$};
 \end{tikzpicture}
\caption{Geometric interpretation of the function $\tau_S$ associated to the set $S$. \label{fig:tauS}}
\end{figure}

\begin{proof}
By Remark \ref{rem:shifts}, it will be sufficient to show that $S \subset I^2$ is a bounded remainder set if and only if $\tau_S(x+x_0)$ is a bounded remainder function for some shift $x_0 \in I$. We will verify this for $x_0=\{x_2-x_1 \alpha \}$. 

Recall from Definition \ref{def:cbrs} that $S$ is a bounded remainder set if the difference
\begin{equation*}
 \Delta_T(S, \alpha, \vx) = \int_0^T \chi_S \left( \{x_1+t\}, \{x_2+t \alpha \} \right) \, dt - T \lambda (S)
\end{equation*}
is uniformly bounded in $T$. For a given $T>0$ we let $N = \floor{T}$, and denote by $S_N(\alpha, x_0)$ the difference 
\begin{equation*}
 S_N(\alpha, x_0) = \sum_{k=0}^{N-1} \tau_S \left( \{ k \alpha + x_0 \}\right) - N \lambda (S) .
\end{equation*}
By Definition \ref{def:brf}, the function $\tau_S(x+x_0)$ is of bounded remainder if $S_N(\alpha, x_0)$ is bounded uniformly in $N$. Thus, to prove Lemma \ref{lem:equivs} it is sufficient to show that 
\begin{equation}
 \label{eq:boundeddiff}
 \left| S_N(\alpha, x_0) - \Delta_T(S, \alpha, \vx) \right| \leq C ,
\end{equation}
where $C$ is a constant independent of $T$ (or equivalently, of $N$).

To verify \eqref{eq:boundeddiff}, we observe that
\begin{equation*}
 \begin{aligned}
  S_N(\alpha, x_0) &= \sum_{k=0}^{N-1} \int_0^1 \chi_S \left( t, \{(t+k)\alpha + x_0 \}\right) \, dt - N \lambda (S) \\
  &= \int_0^N \chi_S \left( \{t\}, \{t\alpha + x_0 \} \right) \, dt - N \lambda (S) \\
  &= \int_{-x_1}^{\floor{T}-x_1} \chi_S \left( \{x_1+t\}, \{x_2+t\alpha\}\right) \, dt - \floor{T} \lambda (S) .
 \end{aligned}
\end{equation*}
It is now easy to see that the difference in \eqref{eq:boundeddiff} must be bounded by 
 \begin{equation*}
  \begin{aligned}
   \left| \int_{-x_1}^0 \chi_S \left( \{x_1+t\}, \{x_2+t\alpha\}\right) \, dt \right|  &+ \\
   \left| \int_{\floor{T}-x_1}^T \chi_S \left( \{x_1+t\}, \{x_2+t\alpha\}\right) \, dt \right| &+ \{T\} \lambda (S) \leq 4 ,
  \end{aligned}
 \end{equation*}
thus verifying \eqref{eq:boundeddiff} and completing the proof of Lemma \ref{lem:equivs}.
\end{proof}

With Lemma \ref{lem:equivs} established, we are equipped to prove Theorem \ref{thm:mainpoly}.
\begin{proof}[Proof of Theorem \ref{thm:mainpoly}]
It will be sufficient to consider the special case when $S$ is a triangle. This is easy to see when $S$ is a convex polygon; $S$ can then be partitioned into finitely many triangles which are disjoint (up to boundaries), and which all have the property that no edge has slope $\alpha$. Finally, since any union of finitely many disjoint bounded remainder sets is again a bounded remainder set for the irrational rotation with slope $\alpha$, the result follows. A similar, but slightly more involved argument can be given to show that also the case when $S$ is non-convex follows from the triangle case.
We thus aim to prove that for almost all $\alpha>0$ and every $\mathbf{x} \in I^2$, every triangle $S$ with no edge of slope $\alpha$ is a bounded remainder set for the continuous irrational rotation with slope $\alpha$ and starting point $\mathbf{x}$.

Fix some $\alpha$, and let $S$ be a triangle with no edge of slope $\alpha$. Denote by $l(y)$ the intersection in the plane of $S$ and the straight line with slope $\alpha$ through the point $(0,y)$, and let $T_S \, : \, \br \mapsto [0,\infty)$ be the function 
\begin{equation*}
 T_S(y) = \frac{|l(y)|}{\sqrt{1+\alpha^2}} .
\end{equation*} 
Then $T_S$ is a (possibly shifted) hat function as defined in \eqref{eq:hatfnc} and $\tau_S$ in \eqref{eq:tau} is given by
 \begin{equation*}
  \tau_S(x) = \sum_{m \in \bz} T_S(x+m) .
 \end{equation*}

Let $\vx \in I^2$ be any given starting point for the irrational rotation. By Lemma \ref{lem:equivs}, the triangle $S$ is a bounded remainder set if and only if $\tau_S$ is a bounded remainder function with respect to $\alpha$. By Proposition \ref{prop:hatfncbrf}, this is indeed the case for every irrational $\alpha>0$ whose continued fraction expansion satisfies
\begin{equation}
\label{eq:alphcond}
 \sum_{l=0}^s \frac{a_{l+1}}{q_{l}^{1/2}} \sum_{k=1}^{l+1}a_k \leq C
\end{equation}
for some constant $C$ independent of $s$, i.e.\ a set of full measure. This completes the proof of Theorem \ref{thm:mainpoly}. 
\end{proof}

We complete this section with the proof of Theorem \ref{thm:mainconvex}. Recall that this result says that for every $\mathbf{x} \in I^2$ and almost all $\alpha>0$, every convex set $S$ whose boundary is a twice differentiable curve with positive curvature at every point is a bounded remainder set for the continuous irrational rotation with slope $\alpha$ and starting point $\mathbf{x}$.
\begin{proof}[Proof of Theorem \ref{thm:mainconvex}]
 We have seen in Lemma \ref{lem:equivs} that the set $S$ is of bounded remainder for the irrational rotation with slope $\alpha$ and starting point $\vx \in I^2$ if and only if the associated function $\tau_S$ in \eqref{eq:tau} is of bounded remainder with respect to $\alpha$. Suppose that $\tau_S$ is of the form 
 \begin{equation}
 \label{eq:csformoftau}
  \tau_S(x) = \sum_{m \in \bz} T_S(x+m),
 \end{equation}
where $T_S$ is the shift of a dome function as given in Definition \ref{def:domefnc}. Then this would be an immediate consequence of Proposition \ref{prop:domefncbrf} and Remark \ref{rem:shifts} for every $\mathbf{x} \in I^2$ and every irrational $\alpha >0$ satisfying \eqref{eq:alphcond}. Our proof is thus complete if we can show that $\tau_S$ is of the form \eqref{eq:csformoftau} for some shifted dome function $T_S$.

As in the proof of Theorem \ref{thm:mainpoly}, we let $l(y)$ be the intersection in the plane of the set $S$ and the straight line with slope $\alpha$ through the point $(0,y)$, and we let $T_S \, : \, \br \mapsto [0, \infty)$ be the function
\begin{equation*}
 T_S(y) = \frac{|l(y)|}{\sqrt{1+\alpha^2}} .
\end{equation*}
Then $\tau_S$ is given in \eqref{eq:csformoftau}. It is clear that $T_S$ is a continuous function supported on some interval $[B_1, B_2]$, and that an appropriate shift of $T_S$ would satisfy condition \eqref{it:domediff} in Definition \ref{def:domefnc}. We will show that also condition \eqref{it:domegrowth} is satisfied for this shift of $T_S$; that is, we can find $c>0$, $m>0$ and $\varepsilon>0$ such that
\begin{equation*}
 T_S(B_1+x) \leq cx^{1/m}, 
\end{equation*}
and
\begin{equation*}
 T_S(B_2-x) \leq cx^{1/m},
\end{equation*}
whenever $0 \leq x < \varepsilon$. We verify only the latter inequality (the argument for the former is equivalent).

Let $C=(C_1(s), C_2(s))$ denote the boundary of $S$ parametrized by arc length, and denote by $L$ its total length. We then have $|C'(s)|=1$ and $C'(s) \perp C''(s)$ for all $s \in [0,L]$. The curvature $\kappa (s)$ at the point $C(s)$ is given by $\kappa (s) = |C''(s)|$, and assumed positive for all $s \in [0,L]$. We let
\begin{equation}
\label{eq:defkappa}
k := \min_{s \in [0,L]} \kappa (s) .
\end{equation}

The line with slope $\alpha$ through the point $(0,B_2)$ in the plane will intersect the curve $C$ at a single point $p$. We let this line be the $x$-axis in a new coordinate system $(x,y)$ where $p$ is the origin (see Figure \ref{fig:defH}), and view $C$ as a curve in these coordinates with $C(0) = (C_1(0), C_2(0)) = (0,0)$. 
\begin{figure}
 \centering
 \begin{tikzpicture}[scale=1]
 %lines/axes and ellipses
 \draw (-.5,0) -- (5,0);
 \draw (0,-.5) -- (0,5);
 \draw[blue] (2, 2.5) ellipse (1cm and 1.5cm);
 \draw[gray,semithick,latex-] (-.5, {(10+3*sqrt(3))/4-sqrt(3)}) -- (3.5, {(10+3*sqrt(3))/4+sqrt(3)});
 \draw[gray,semithick,-latex] (.5, {(10+3*sqrt(3))/4+2/sqrt(3)}) -- (3.5, {(10+3*sqrt(3))/4-4/sqrt(3)});
 %points and notation
 \fill (1.5, {(10+3*sqrt(3))/4}) circle (.05);
 \draw (1.5, {(10+3*sqrt(3))/4+.07}) node[above]{$p$};
 \fill (0, {10/4}) circle (.05);
 \draw (0, {10/4+.1}) node[left]{$B_2$};
 \draw[gray] (-.5, {(10+3*sqrt(3))/4-sqrt(3)}) node[below left]{$x$};
 \draw[gray] (3.5, {(10+3*sqrt(3))/4-4/sqrt(3)}) node[below right]{$y$};
 \draw[blue] (2.9,3.9) node{$C$};
 \end{tikzpicture}
\caption{The curve $C$ and the new coordinate axes $x$ and $y$. \label{fig:defH}}
\end{figure}
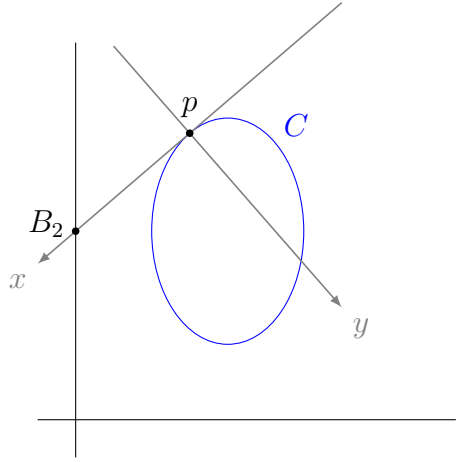
We may then think of a section of $C$ around $p$ as the graph of the function $H \, : \, (-\delta, \delta) \mapsto [0,\infty)$ given by
\begin{equation*}
 H(x) = C_2\left( C_1^{-1}(x)\right) .
\end{equation*}
We have $C_2'(0)=0$ and $C_1'(0)=1$, and since $C_1$ and $C_2$ are both twice continuously differentiable it follows that 
\begin{equation*}
 H'(x) = \frac{C_2'(C_1^{-1}(x))}{C_1'(C_1^{-1}(x))} 
\end{equation*}
and
\begin{equation*}
 H''(x) = \frac{C_2''(s)C_1'(s)-C_2'(s)C_1''(s)}{\left( C_1'(s) \right)^3}, \quad s=C_1^{-1}(x),
\end{equation*}
are both well-defined and continuous on some interval $(-\delta, \delta)$. By choosing $\delta$ sufficiently small we can ensure that 
\begin{equation*}
 \left| C_1'(C_1^{-1}(x)) \right| \geq \frac{1}{2}, \quad x \in (-\delta, \delta),
\end{equation*}
which (for $s=C_1^{-1}(x)$ and recalling that $C'(s) \perp C''(s)$) in turn implies 
\begin{equation}
 \label{eq:boundHprime2}
 |H''(x)| = \frac{|C''(s)| \cdot |C'(s)|}{|C_1'(s)|^3} \geq \frac{k}{8} , \quad x \in (-\delta, \delta),
\end{equation}
with $k$ given in \eqref{eq:defkappa}. 

We now use this lower bound on $|H''(x)|$ to find an upper bound on $T_S(B_2-z)$ for sufficiently small $z>0$. We have 
\begin{equation}
\label{eq:TSl}
T_S(B_2-z)=\frac{|l|}{\sqrt{1+\alpha^2}} ,
\end{equation}
where $l$ is the intersection of $S$ with the line of slope $\alpha$ through the point $(0,B_2-z)$, illustrated in Figure \ref{fig:defell}. 
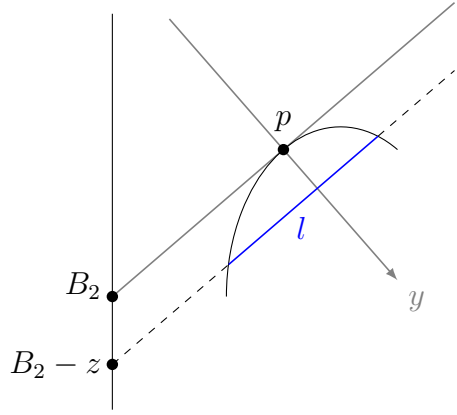
\begin{figure}
 \centering
 \begin{tikzpicture}[scale=1.5]
%Arcs and lines
 \draw (0,1.5) -- (0,5);
  \draw[gray,semithick] (0, {10/4}) -- (3, {(10+3*sqrt(3))/4+3/4*sqrt(3)});
   \draw[gray,semithick,-latex] (.5, {(10+3*sqrt(3))/4+2/sqrt(3)}) -- (2.5, {(10+3*sqrt(3))/4-2/sqrt(3)});
  \draw (2.5, {(10+3*sqrt(3))/4}) arc (60:180:1cm and 1.5cm);
  \draw[dashed] (0, {10/4-.6}) -- (1.02, {10/4-.6+1.02*sqrt(3)/2});
  \draw[blue, semithick] (1.02, {10/4-.6+1.02*sqrt(3)/2}) -- (2.33, {10/4-.6+2.33*sqrt(3)/2});
  \draw[dashed] (2.33, {10/4-.6+2.33*sqrt(3)/2}) -- (3, {(10+3*sqrt(3))/4-.6+3/4*sqrt(3)});
%points and notation
 \fill (1.5, {(10+3*sqrt(3))/4}) circle (.05);
 \draw (1.5, {(10+3*sqrt(3))/4+.07}) node[above]{$p$};
 \fill (0, {10/4}) circle (.05);
 \draw (0, {10/4+.1}) node[left]{$B_2$};
  \fill (0, {10/4-.6}) circle (.05);
 \draw (0, {10/4-.6}) node[left]{$B_2-z$};
 \draw[blue] (1.8, {(10+3*sqrt(3))/4-.3*2/sqrt(3)-.15}) node[below left]{$l$};
 \draw[gray] (2.5, {(10+3*sqrt(3))/4-2/sqrt(3)}) node[below right]{$y$};
 \end{tikzpicture}
 \caption{The intersection $l$ of $S$ and the line of slope $\alpha$ through the point $(0,B_2-z)$. \label{fig:defell}}
\end{figure}
The line segment $l$ is at height $y=z/\sqrt{1+\alpha^2}$ above the point $p$ (see Figure \ref{fig:defell}). If $y < \min \{ H(\delta), H(-\delta) \}$, then we may denote by $x_1, x_2$ the two values of $x \in (-\delta, \delta)$ satisfying $H(x)=y$, and
\begin{equation}
 \label{eq:absl}
 |l| \leq 2 \max \left\{ |x_1|, |x_2| \right\} .
\end{equation}
By Taylor's theorem we have
\begin{equation*}
y = H(x_i) =H(0) + \frac{H'(0)}{1!}x_i + \frac{H''(r_i)}{2!}x_i^2 = \frac{H''(r_i)}{2!}x_i^2 , 
\end{equation*}
for $i=1,2$ and some $r_i \in (-\delta, \delta)$, and from \eqref{eq:boundHprime2} it thus follows that 
\begin{equation*}
 |x_i| = \left( \frac{2y}{H''(r_i)} \right)^{1/2} \leq \frac{4}{\sqrt{k}} \cdot y^{1/2}, \quad  i=1,2 .
\end{equation*}
Hence, from \eqref{eq:absl} we get 
\begin{equation}
\label{eq:boundonell}
|l| \leq \frac{8}{\sqrt{k}} \cdot y^{1/2} \leq \frac{8}{\sqrt{k}} \cdot z^{1/2},
\end{equation}
and by \eqref{eq:TSl} and \eqref{eq:boundonell} we have
\begin{equation*}
 T_S(B_2-z) \leq \frac{8}{\sqrt{k(1+\alpha^2)}} \cdot z^{1/2} .
\end{equation*}
This verifies that a shift of the function $T_S$ satisfies the growth condition \eqref{it:domegrowth} in Definition \ref{def:domefnc} with $c=8/\sqrt{k(1+\alpha^2)}$, $m=2$ and some $\varepsilon >0$ (for instance, $\varepsilon = \min \{ H(\delta), H(-\delta) \}$ will suffice). The function $\tau_S$ is thus of the form \eqref{eq:csformoftau}, where $T_S$ is the shift of a dome function, and this completes the proof of Theorem \ref{thm:mainconvex}.
 \end{proof}

%%%%%%%%%%%%%%%%%%%%%%%%%%%%%%%%% Negative results %%%%%%%%%%%%%%%%%%%%%%%%%%%%%%%%%%%%%%%

\section{Proof of Theorem \ref{thm:mainneg}}
\label{sec:neg}
In this section we present the proof of Theorem \ref{thm:mainneg}. 
For the proof of part \eqref{it:negpoly} we simply give an outline, as this proof largely follows the proof given above for Proposition \ref{prop:hatfncbrf}.
Part \eqref{it:negconvex}, on the other hand, is proven in full detail. Lastly, we present the proof of part \eqref{it:specset}.

\begin{proof}[Proof of Theorem \ref{thm:mainneg} \ref{it:negpoly}]
Fix an irrational $\alpha>0$ with continued fraction expansion $\alpha = [0; a_1, a_2, a_3, \cdots]$ satisfying $a_1=1$ and $a_{l+1}\geq q_l^7$. One can show that there are uncountably many such irrationals. 

Let $S$ be the triangle with vertices $(0,0)$, $(0,1)$ and $(K,1)$ for some $0<K<1$ to be determined. We will assume that $1-K\alpha >0$. Denote by $\tau_S$ the function in \eqref{eq:tau} associated to $S$; this is a hat function as defined in \eqref{eq:hatfnc}, with $a=1-K\alpha$ and $b=1$. By Lemma \ref{lem:equivs}, the triangle $S$ is a bounded remainder set for the continuous irrational rotation with slope $\alpha$ and some arbitrary starting point $\mathbf{x} \in I^2$ if and only if $\tau_S$ is a bounded remainder function with respect to $\alpha$. In what follows we show that the latter is \emph{not} the case, and accordingly $S$ is not a bounded remainder set.

For $N=\sum_{l=0}^s b_l q_l$, one can show by calculations analogous to those in the proof of Proposition \ref{prop:hatfncbrf} that
\begin{equation}
 \label{eq:negpoly1}
\left| \sum_{k=0}^{N-1} \tau_S(\{k \alpha\})- \frac{NK}{2} \right| = C \sum_{l=0}^s \xi_l (1-\xi_l) \frac{b_l}{q_l} + O(1) , 
\end{equation}
where $C$ depends only on $K$ and $\alpha$, and $\xi_l =\{q_l a\} = \{q_l(1-K\alpha)\}$. For $x \in \br$, let $\norm{x}$ denote the minimal distance from $x$ to an integer, and note that
\begin{equation*}
 \xi_l (1-\xi_l) \geq \frac{1}{2} \norm{q_la}.  
\end{equation*}
It is a well-known fact (see e.g.\ \cite[p.\ 69]{khinchin}) that for almost all $a \in (0,1)$ one can find a positive constant $c$ such that 
\begin{equation*}
 \norm{n\cdot a} \geq \frac{c}{n^2}
\end{equation*}
for all $n\geq 2$.
Thus, one can indeed find $K\in (0,1)$ such that $a= 1-K\alpha >0$, and moreover
\begin{equation*}
 C \sum_{l=0}^s \xi_l (1-\xi_l) \frac{b_l}{q_l} \geq C \sum_{l=0}^s \norm{q_la} \frac{b_l}{q_l} > C \sum_{l=0}^s \frac{b_l}{q_l^3}.
\end{equation*}

Now let $b_l:=q_l^4$. Then the sum on the right hand side in \eqref{eq:negpoly1} is bounded from below by $C \sum_{l=0}^s q_l$, which tends to infinity as $s \rightarrow \infty$. For the sequence of integers $N_s = \sum_{l=0}^s q_l^5$, we thus have 
\begin{equation*}
\left| \sum_{k=0}^{N_s-1} \tau_S(\{k \alpha\})- N_s \lambda(S) \right| \rightarrow \infty  
\end{equation*}
as $s \rightarrow \infty$. This shows that $\tau_S$ is not a bounded remainder function with respect to $\alpha$, and completes the proof of Theorem \ref{thm:mainneg} \ref{it:negpoly}.
\end{proof}

\begin{proof}[Proof of Theorem \ref{thm:mainneg} \ref{it:negconvex}]
Fix an irrational $\alpha \in (1/4, 1/2)$ with continued fraction expansion $\alpha=[0;a_1, a_2, \ldots ]$ satisfying $a_{l+1}>q_l^{100}$ and $p_l$ even for an infinite number of odd indices $l$, say for the sequence $l_1 < l_2 < l_3 \ldots$. One can show that there are uncountably many such irrationals.

Let $S$ be the disc with diameter $d:=\alpha / \sqrt{1+\alpha^2}$ illustrated in Figure \ref{fig:negdisc}. By Lemma \ref{lem:equivs}, the set $S$ is of bounded remainder for the continuous irrational rotation with slope $\alpha$ and arbitrary starting point $\mathbf{x} \in I^2$ if and only if the associated function $\tau_S$ in \eqref{eq:tau} is a bounded remainder function with respect to $\alpha$. In what follows, we will show that there exists an $x \in I$ and a sequence of integers $N_1 < N_2 < N_3 \ldots$ such that 
\begin{equation*}
\left| \sum_{k=0}^{N_i-1} \tau_S \left( \{ k\alpha +x \} \right) - N_i \lambda (S) \right| \rightarrow \infty
\end{equation*}
as $i \rightarrow \infty$. By Remark \ref{rem:shifts}, this proves $\tau_S$ is not a bounded remainder function, and accordingly $S$ is not a bounded remainder set.
\begin{figure}[htb]
 \centering
 \begin{tikzpicture}[scale=5]
%axes
\draw (-.15,0) -- (1.15, 0);
\draw (0,-.1) -- (0, 1.1);
%limits of unit square
\draw[dashed] (0,1) -- (1,1) -- (1,0);
\draw(0,1) node[left]{$1$};
\draw(1,0) node[below]{$1$};
%the set S
\draw[draw=black, fill={black!20!white}] ({2/3}, {1/(3*sqrt(2))+1/(4*sqrt(2))}) circle[radius={sqrt(3)/(4*sqrt(7)}];
\draw ({2/3}, {1/(3*sqrt(2))+1/(4*sqrt(2))}) node{$S$};
%band lines
\draw (0,0) -- (1,{1/(2*sqrt(2))});
\draw (0,{1/(2*sqrt(2))}) -- (1, {2/(2*sqrt(2))}); 
\draw[dashed] (0,{1/(2*sqrt(2))}) -- (1,{1/(2*sqrt(2))});
%arrows
\draw[latex-latex] (-.1, 0) -- (-.1, {1/(2*sqrt(2))});
\draw(-.1, {1/(4*sqrt(2))}) node[left]{$\alpha$};
\draw[latex-latex] (1.1,0) -- (1.1, {1/(2*sqrt(2))});
\draw(1.1,{1/(4*sqrt(2))}) node[right]{$\alpha$};
 \end{tikzpicture}
 \caption{The disc $S$ with diameter $d=\alpha / \sqrt{1+\alpha^2}$. \label{fig:negdisc}}
\end{figure}
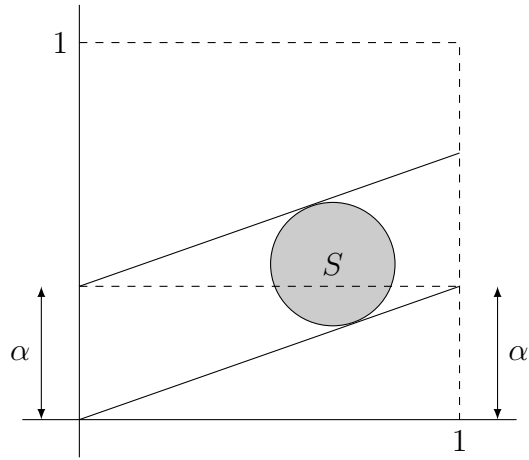

The function $\tau_S$ associated to $S$ is given explicitly by
\begin{equation*}
 \tau_S (y) = \begin{cases}
               \frac{\alpha}{1+\alpha^2} \sqrt{1-(1-2y/\alpha)^2} \text{ ,} &0\leq y\leq \alpha \text{ ;}\\
               0 \text{ ,} &\alpha<y\leq 1 \text{ ,}
              \end{cases}
\end{equation*}
and we note that 
\begin{equation}
\label{eq:negdiscmes}
 \lambda(S) = \int_0^1 \tau_S (y) \, dy = \frac{\pi}{4} \cdot \frac{\alpha^2}{1+\alpha^2} .
\end{equation}
We introduce the notation $S_N(x)$ for the sum
\begin{equation*}
 S_N(x) := \sum_{k=0}^{N-1} \tau_S \left( \{ k\alpha +x\}\right) .
\end{equation*}
Let us now fix some $i$ (and thereby an odd index $l_i$), put 
\begin{equation*}
 p:=p_{l_i}=2m \quad (m \in \bn), \quad q:= q_{l_i},
\end{equation*}
and evaluate the sum $S_N(x)$ for $N:=q^{11}$ and some $x \in [0,1/q]$. We then have 
\begin{equation}
 \label{eq:SNdouble}
 S_N(x) = \sum_{j=0}^{q^{10}-1} \sum_{k=0}^{q-1} \tau_S \left( \{ (jq+k)\alpha + x \}\right) .
\end{equation}
Recall from \eqref{eq:cferror} that 
\begin{equation*}
 \left| \alpha - \frac{p}{q} \right| \leq \frac{1}{q^2 a_{l_i+1}} \leq \frac{1}{q^{102}}.
\end{equation*}
Using this fact, we get 
\begin{equation*}
 \norm{jq\alpha} < j \norm{q\alpha} \leq \frac{j}{q^{101}} < \frac{1}{q^{91}},
\end{equation*}
where $\norm{x}$ denotes the minimal distance from $x \in \br$ to an integer. It follows that 
\begin{equation*}
\begin{aligned}
 \norm{\left\{ (jq+k)\alpha + x \right\} - \left\{ k \cdot \frac{p}{q} + x\right\}} &\leq  \norm{jq\alpha}+k\norm{\alpha-\frac{p}{q}} \\
 &< \frac{1}{q^{91}}+\frac{q}{q^{102}} < \frac{1}{q^{90}} ,
\end{aligned}
\end{equation*}
and hence
\begin{equation*}
 \left| \tau_S \left( \{(jq+k)\alpha+x\} \right) - \tau_S \left( \left\{ k\cdot \frac{p}{q}+x \right\} \right) \right| \leq \left| \tau_S \left( \frac{1}{q^{90}} \right) \right| < \frac{1}{q^{44}}.
\end{equation*}
Combining this bound with \eqref{eq:SNdouble}, we get
\begin{equation}
\label{eq:SNbound1}
\left| S_N (x) - q^{10}\sum_{k=0}^{q-1} \tau_S \left( \left\{ k \cdot \frac{p}{q}+x\right\} \right)\right| < q^{11} \cdot \frac{1}{q^{44}} = \frac{1}{q^{33}} .
\end{equation}
In light of \eqref{eq:SNbound1}, we introduce the function 
\begin{equation*}
 \sigma (y) := \begin{cases}
               \frac{\alpha}{1+\alpha^2} \sqrt{1-(1-2qy/p)^2} \text{ ,} &0\leq y\leq p/q \text{ ;}\\
               0 \text{ ,} &p/q<y\leq 1 \text{ .}
              \end{cases}
\end{equation*}
Since the index $l_i$ is odd, we have $\alpha < p/q$ and $\sigma (y)=\tau_S (\alpha qy/p)$ for all $y \in [0,1)$. From 
\begin{equation*}
 \left| \frac{\alpha q}{p} -1 \right| = \frac{q}{p} \left| \alpha - \frac{p}{q} \right| < \frac{1}{\alpha} \cdot \frac{1}{q^{102}} < \frac{1}{q^{101}}
\end{equation*}
it thus follows that
\begin{equation*}
 \left| \sigma (y) - \tau_S(y) \right|= \left| \tau_S \left( \frac{\alpha p}{q} y \right) - \tau_S(y) \right| < \left| \tau_S \left( \frac{1}{q^{101}} \right) \right| < \frac{1}{q^{50}} .
\end{equation*}
Combining this bound with \eqref{eq:SNbound1}, we get
\begin{equation}
 \label{eq:SNbound2}
 \left| S_N (x) - q^{10}\sum_{k=0}^{q-1} \sigma \left( \left\{ k \cdot \frac{p}{q}+x\right\} \right)\right| <  \frac{1}{q^{33}}+ \frac{q^{11}}{q^{50}} < \frac{1}{q^{32}} .
\end{equation}
Note that some of the above estimates hold only for $q$ greater than some lower threshold $q>q_0$.

Let us now have a closer look at the sum over $\sigma$ in \eqref{eq:SNbound2}. We have
\begin{equation}
\label{eq:sigmabound}
\begin{aligned}
 \sum_{k=0}^{q-1} \sigma \left( \left\{ k \cdot \frac{p}{q}+x\right\} \right) &= \sum_{k=0}^{p-1} \sigma \left( \frac{k}{q} +x\right) \\
 &=\frac{\alpha}{1+\alpha^2} \sum_{k=0}^{p-1} \sqrt{1- \left( 1-\frac{2k}{p}-\frac{2q}{p} x \right)^2} \\
 &=\frac{\alpha}{1+\alpha^2} \sum_{k=0}^{2m-1} \sqrt{1- \left( 1-\frac{k}{m}-\frac{q}{m} x \right)^2} \\
 &=\frac{\alpha}{1+\alpha^2} \cdot 2mG_m\left( \frac{q}{m} x \right),
 \end{aligned}
\end{equation}
where
\begin{equation*}
 G_m(x) := \frac{1}{2m} \sum_{k=0}^{2m-1} \sqrt{1-\left( 1-\frac{k}{m} - x\right)^2} \text{ ,} \quad x \in \left[0, \frac{1}{m}\right) .
\end{equation*}
\begin{figure}
 \centering
\begin{tikzpicture}[scale=3]
%axes
\draw[-latex] (-.15,0) -- (2.3, 0);
\draw (2.3,0) node[below]{$x$};
\draw[-latex] (0,-.1) -- (0, 1.4);
\draw (0,1.2) node[left]{$G_m(x)$};
\draw (1,-.03) -- (1, .03);
\draw(1,-.03) node[below]{$1/2m$};
\draw (2,-.03) -- (2, .03);
\draw(2,-.03) node[below]{$1/m$};
%arc
\draw (2,0.3) arc[radius=1, start angle=0, end angle=180]; 
\end{tikzpicture}
\caption{The function $G_m(x)$. \label{fig:Gm}}
\end{figure}
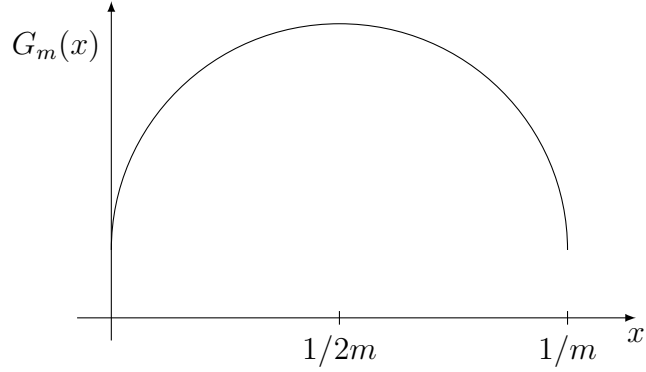
The function $G_m$ is illustrated in Figure \ref{fig:Gm}. It is clear that $G_m(x) = G_m(1/m-x)$, and by elementary analysis one can show that $G_m$ increases on $[0,1/(2m))$ in such a way that 
\begin{equation*}
 G_m \left( \frac{1}{3m} \right) > G_m \left( \frac{1}{6m} \right) + \frac{2c}{m^{3/2}}
\end{equation*}
for some $c>0$. From this inequality one can deduce that there exists a subinterval $\Lambda \subset [0,1/(2m)]$ of length at least $1/(6m)$ such that either
\begin{equation}
 \label{eq:Gcond1}
 G_m(x) > \frac{1}{2} \int_0^2 \sqrt{1-(1-y)^2} \, dy + \frac{c}{m^{3/2}} = \frac{\pi}{4} + \frac{c}{m^{3/2}}
\end{equation}
or
\begin{equation}
 \label{eq:Gcond2}
 G_m(x) < \frac{1}{2} \int_0^2 \sqrt{1-(1-y)^2} \, dy - \frac{c}{m^{3/2}} = \frac{\pi}{4} - \frac{c}{m^{3/2}}
\end{equation}
for all $x \in \Lambda$. We assume in what follows that \eqref{eq:Gcond1} holds for all $x \in \Lambda$ (the case when \eqref{eq:Gcond2} holds is treated similarly). Then for $\tilde{x} \in \tilde{\Lambda}$, where
\begin{equation*}
\tilde{\Lambda}:= (m/q)\Lambda \subset [0,1/(2q)) , 
\end{equation*}
we have $q\tilde{x}/m \in \Lambda$, and from \eqref{eq:sigmabound} and \eqref{eq:Gcond1} it follows that 
\begin{equation}
\label{eq:sigmabound2}
 \sum_{k=0}^{q-1} \sigma \left( \left\{ k \cdot \frac{p}{q}+\tilde{x} \right\} \right) > \frac{\alpha}{1+\alpha^2} \cdot 2m \left( \frac{\pi}{4}+\frac{c}{m^{3/2}} \right).
\end{equation}
In the following we let $c_1, c_2, \ldots$ denote positive absolute constants. From \eqref{eq:sigmabound2} and \eqref{eq:SNbound2} we get
\begin{equation*}
\begin{aligned}
 S_N(\tilde{x}) &> q^{10} \sum_{k=0}^{q-1} \sigma \left( \left\{ k \cdot \frac{p}{q}+\tilde{x} \right\} \right) - \frac{1}{q^{32}} \\
 &> q^{10} \cdot 2m \cdot \frac{\alpha}{1+\alpha^2} \left( \frac{\pi}{4}+\frac{c}{m^{3/2}}\right) - \frac{1}{q^{32}} \\
 &> N \cdot \frac{p}{q}\cdot \frac{\pi\alpha}{4(1+\alpha^2)} + c_1q^9 \\
 &> N \cdot \frac{\pi \alpha^2}{4(1+\alpha^2)} + c_1q^9 = N \lambda(S) + c_1 q^9,
\end{aligned}
\end{equation*}
where we recall from \eqref{eq:negdiscmes} that $\lambda (S)$ is the integral over $\tau_S$ and the measure of the disc $S$ in Figure \ref{fig:negdisc}. Thus, we have shown that 
\begin{equation}
\label{eq:SNtilde}
 S_N(\tilde{x}) - N \lambda (S) > c_1q^9 \text{ ,} \quad \tilde{x} \in \tilde{\Lambda} .
\end{equation}

Finally, we define the set $\bar{\Lambda} \subset I$ by
\begin{equation*}
 \tilde{\Lambda}^{(j)} := \tilde{\Lambda} + \frac{j}{q} \text{ ,} \quad \bar{\Lambda}:=  \bigcup_{j=0}^{q-1} \tilde{\Lambda}^{(j)} .
\end{equation*}
Since $\lambda(\tilde{\Lambda}) \geq 1/(6q)$, we have $\lambda (\bar{\Lambda}) \geq 1/6$. Choose some $x \in \bar{\Lambda}$, and find $j \in \{0,1, \ldots, q-1\}$ such that 
\begin{equation*}
 x=\tilde{x}+\frac{j}{q} \text{ ,} \quad \tilde{x} \in \tilde{\Lambda} .
\end{equation*}
Furthermore, choose $k_j \in \{0,1, \ldots , q-1\}$ such that $k_jp \equiv q-j (\modu q)$, and note that
\begin{equation*}
\norm{k_j\alpha+\frac{j}{q}} = \norm{k_j\alpha - \frac{k_jp}{q}} \leq k_j \norm{\alpha - \frac{p}{q}} < \frac{1}{q^{101}} . 
\end{equation*}
From this and the fact that $|\tau_S|\leq 1$, it follows that
\begin{equation*}
 \begin{aligned}
  S_N(x) &> \sum_{k=0}^{k_j} \tau_S \left( \{ k\alpha +x\} \right) + \sum_{k=0}^{N-1} \tau_S \left(\{k\alpha + k_j \alpha +x\} \right) -q \\
  &> \sum_{k=0}^{N-1} \tau_S \left( \left\{ k\alpha + x -\frac{j}{q}\right\} \right) - c_2q \\
  &= \sum_{k=0}^{N-1} \tau_S \left( \{k\alpha + \tilde{x} \}\right) - c_2q = S_N(\tilde{x})-c_2q ,
 \end{aligned}
\end{equation*}
and from \eqref{eq:SNtilde} we thus get
\begin{equation}
 \label{eq:SNfinal}
 S_N(x) - N \lambda (S) > c_3 q^9
\end{equation}
for all $x \in \bar{\Lambda}$. 

The above analysis can be carried out for each $l_i$ (given that $q_{l_i}$ is above the threshold $q_{l_i}>q_0$). That is, for each $i$, we find $\bar{\Lambda}_i \subset I$ of measure $\lambda(\bar{\Lambda}_i)\geq 1/6$ such that \eqref{eq:SNfinal} holds for all $x \in \bar{\Lambda}_i$ with $q=q_{l_i}$ and $N=q^{11}$. Now fix $x \in I$ such that $x \in \bar{\Lambda}_i$ for infinitely many $i$, and for each such $i$ let $q_i = q_{l_i}$ and $N_i = q_i^{11}$. Then for these $N_i$, we have
\begin{equation*}
 \left| S_{N_i}(x)-N_i \lambda (S) \right| = \left| \sum_{k=0}^{N_i-1} \tau_S \left( \{k \alpha+x\}\right)-N_i \lambda (S)\right| \rightarrow \infty 
\end{equation*}
as $i \rightarrow \infty$. This verifies that $\tau_S$ is not a bounded remainder function with respect to $\alpha$, and completes the proof of Theorem \ref{thm:mainneg} \ref{it:negconvex}.
\end{proof}

\begin{proof}[Proof of Theorem \ref{thm:mainneg} \ref{it:specset}]
 Let $S$ be the triangle with vertices $(0,0)$, $(0,1)$ and $(1,0)$. Fix some slope $\alpha>0$ and starting point $\mathbf{x} \in I^2$. For simplicity we assume that $\alpha<1$ (the proof is similar when $\alpha\geq 1$). By Lemma \ref{lem:equivs}, the set $S$ is of bounded remainder for the continuous irrational rotation with slope $\alpha$ and starting point $\mathbf{x}$ if and only if the associated function $\tau_S$ in \eqref{eq:tau} is of bounded remainder with respect to $\alpha$. For the specific triangle $S$, we have
 \begin{equation}
  \label{eq:specsettau}
  \tau_S(x) = \begin{cases}
               \frac{1-x}{1+\alpha} \text{ ,} & 0 \leq x \leq 1-\alpha \text{ ;}\\
               \frac{1-x}{1+\alpha} + \frac{2-x}{1+\alpha}-\frac{1-x}{\alpha} \text{ ,} & 1-\alpha <x\leq 1 \text{ .}
              \end{cases}
 \end{equation}
 
 It is a well-known fact that a $1$-periodic function $f$ which is integrable over the unit interval $I$ is a bounded remainder function with respect to $\alpha$ if and only if there exists a bounded and measurable $1$-periodic function $g$ satisfying the equation 
 \begin{equation*}
  f(x) - \int_0^1 f(t) \, dt = g(x) - g(x+\alpha) 
 \end{equation*}
 for almost every $x$. This is known as the \emph{cohomological equation} for $f$. By a classical result of Gottschalk and Hedlund \cite[Theorem 14.11]{gottschalk}, the function $g$ can be chosen to be continuous whenever $f$ is continuous. Thus, our proof is complete if we can find a continuous $1$-periodic function $g$ such that
 \begin{equation}
  \label{eq:specsetg}
  \tau_S(x) - \int_0^1 \tau_S(t) \, dt = g(x) - g(x+\alpha),
 \end{equation}
where $\tau_S$ is given in \eqref{eq:specsettau}.

Let $g$ be the continuous $1$-periodic function defined on $I$ by
\begin{equation*}
 g(x) = \frac{x(x-1)}{2\alpha(1+\alpha)} .
\end{equation*}
It is straightforward to check that this function satisfies \eqref{eq:specsetg}. This confirms that $\tau_S$ is a bounded remainder function with respect to $\alpha$, and completes the proof of Theorem \ref{thm:mainneg} \ref{it:specset}.
\end{proof}

%%%%%%%%%%%%%%%%%%%%%%%%%%%%%%%%%%%%%% Bibliography %%%%%%%%%%%%%%%%%%%%%%%%%%%%%%%%%%%%%

\end{document}